\documentclass[12pt]{amsart}
\usepackage{amsmath}
\usepackage{amsfonts}
\usepackage{amssymb}
\usepackage[all,cmtip]{xy}           
\usepackage{bm}
\usepackage{bbm}
\usepackage{bbding}
\usepackage{txfonts}
\usepackage{amscd}
\usepackage{xspace}
\usepackage[shortlabels]{enumitem}
\usepackage{ifpdf}
\usepackage{ulem}

\ifpdf
  \usepackage[colorlinks,final,backref=page,hyperindex]{hyperref}
\else
  \usepackage[colorlinks,final,backref=page,hyperindex,hypertex]{hyperref}
\fi
\usepackage{tikz}
\usepackage[active]{srcltx}

\usepackage{tikz-cd}
\usepackage{tikz}

\makeatletter

\newtheorem{thm}{Theorem}[section]
\newtheorem{lem}[thm]{Lemma}
\newtheorem{cor}[thm]{Corollary}
\newtheorem{pro}[thm]{Proposition}
\theoremstyle{definition}
\newtheorem{ex}[thm]{Example}
\newtheorem{rmk}[thm]{Remark}
\newtheorem{defi}[thm]{Definition}

\newcommand{\nc}{\newcommand}
\newcommand{\delete}[1]{}

\nc{\mlabel}[1]{\label{#1}}  
\nc{\mcite}[1]{\cite{#1}}  
\nc{\mref}[1]{\ref{#1}}  
\nc{\meqref}[1]{\eqref{#1}}  
\nc{\mbibitem}[1]{\bibitem{#1}} 

\delete{
\nc{\mlabel}[1]{\label{#1}{\hfill \hspace{1cm}{\bf{{\ }\hfill(#1)}}}}
\nc{\mcite}[1]{\cite{#1}{{\bf{{\ }(#1)}}}}  
\nc{\mref}[1]{\ref{#1}{{\bf{{\ }(#1)}}}}  
\nc{\meqref}[1]{\eqref{#1}{{\bf{{\ }(#1)}}}}  
\nc{\mbibitem}[1]{\bibitem[\bf #1]{#1}} 
}

\newcommand {\emptycomment}[1]{}

\nc{\oprn}{\theta}

\nc{\calo}{\mathcal{O}}
\nc{\oop}{$\mathcal{O}$-operator\xspace}
\nc{\oops}{$\mathcal{O}$-operators\xspace}
\nc{\mrho}{{\bm{\varrho}}}
\nc{\bfk}{\mathbf{K}}
\nc{\invlim}{\displaystyle{\lim_{\longleftarrow}}\,}
\nc{\ot}{\otimes}

\nc{\eval}[1]{\Big|_{#1}}

\newcommand{\be }{\begin{equation}}
\newcommand{\ee }{\end{equation}}

\newcommand{\g}{\mathfrak g}
\newcommand{\h}{\mathfrak h}

\nc{\RR}{\mathbb{R}}

\nc{\CC}{\mathbb{C}}

\newcommand{\huaL}{\mathcal{L}}
\newcommand{\huaR}{\mathcal{R}}

\newcommand{\huaG}{\mathcal{G}}

\newcommand{\huaI}{\mathcal{I}}

\newcommand{\huaO}{{\mathcal{O}}}

\newcommand{\frkc}{\mathfrak c}
\newcommand{\frkd}{\mathfrak d}

\newcommand{\Id}{{\rm{Id}}}

\newcommand{\br}[1]{   [ \cdot,    \cdot  ]   }

\newcommand{\Hom}{\mathrm{Hom}}

\newcommand{\gl}{\mathfrak {gl}}

\newcommand{\ad}{\mathrm{ad}}

\newcommand{\Img}{\mathrm{Im}}

\nc{\CV}{\mathbf{C}}

\newcommand{\hou}[1]{\textcolor{blue}{Hou: #1}}

\begin{document}
\title[Mock-pre-Lie bialgebras]{Mock-pre-Lie bialgebras}

\author{Shuai Hou}
\address{Department of Mathematics, Jilin University, Changchun 130012, Jilin, China}
\email{hshuaisun@jlu.edu.cn}
	
\author{Zafar Normatov}
\address{Department of Mathematics, Jilin University, Changchun 130012, Jilin, China}
\email{z.normatov@inbox.ru}

\author{Lina Song}
\address{Department of Mathematics, Jilin University, Changchun 130012, Jilin, China}
\email{songln@jlu.edu.cn}


\begin{abstract}
  In this paper, we systematically develop the theory of mock-pre-Lie bialgebras from multiple perspectives. We introduce the notion of a phase space of a mock-Lie algebra, and show that a mock-Lie algebra admits a phase space if and only if it is sub-adjacent to a mock-pre-Lie algebra. We introduce the notions of  Manin triples of mock-pre-Lie algebras and mock-pre-Lie bialgebras, and prove the equivalences between mock-pre-Lie bialgebras,  Manin triples of mock-pre-Lie algebras, certain matched pairs of mock-pre-Lie algebras, certain matched pairs of mock-Lie algebras and phase spaces of a mock-Lie algebra, which lays a theoretical foundation for subsequent research. Next, we investigate coboundary mock-pre-Lie bialgebras, and derive an analogue of the classical Yang-Baxter equation. In addition, we introduce two important special classes of mock-pre-Lie bialgebras: quasi-triangular mock-pre-Lie bialgebras and factorizable mock-pre-Lie bialgebras. We show that quasi-triangular mock-pre-Lie bialgebras naturally induce relative Rota-Baxter operators of weight $-1.$ Finally, we provide a new perspective for the study of triangular and factorizable mock-pre-Lie bialgebras by introducing the concept of quadratic Rota-Baxter mock-pre-Lie algebras of arbitrary weight.
\end{abstract}

\subjclass[2010]{17B38, 16T10,17C50}

\keywords{mock-pre-Lie algebra, mock-pre-Lie bialgebra, phase space, Manin triple, classical Yang-Baxter equation}

\maketitle

\tableofcontents

\allowdisplaybreaks


\section{Introduction}\label{sec:intr}
The purpose of this paper is to establish the bialgebra theory for mock-pre-Lie algebras from multiple perspectives,
and to focus on quasi-triangular and factorizable mock-pre-Lie bialgebras.

\subsection{Mock-Lie algebras and mock-pre-Lie algebras}
Mock-Lie algebras \cite{Zusmanovich}, also called Jacobi-Jordan algebras \cite{Burde-Fialowski}, are a class of commutative algebras satisfying the Jacobi identity. Jordan algebras were originally introduced in order to achieve an axiomatization for the algebra of observables in quantum mechanics \cite{JNW}.
 Mock-Lie algebras are a special case of Jordan algebras: a mock-Lie algebra is a Jordan algebra in any characteristic. Furthermore, the class of mock-Lie algebras coincides with the class of commutative nilalgebras of nilindex at most $3$. In addition, mock-Lie algebras can be regarded as a subclass of the Jordan-Lie superalgebras \cite{KO}. In recent years, there has been some related research on mock-Lie algebras.
In \cite{Agore-Militaru-1,Agore-Militaru}, the authors started from arbitrary $3$-nilpotent Jacobi-Jordan algebras to construct a new family of solutions to the quantum Yang-Baxter equation. Beyond that, they also constructed a non-abelian cohomological type object in order to classify certain mock-Lie algebras.
In \cite{Zusmanovich}, the author studied the problems concerning faithful representations of mock-Lie algebras.
The dual algebras of mock-Lie algebras have been investigated in \cite{CKLS}.
The cohomology theory of mock-Lie algebras has been studied in \cite{Benayadi,Yang-Yong}. In \cite{BBMM}, the authors introduced the notion of pre-Jacobi-Jordan algebras, which are also called mock-pre-Lie algebras and have a close relationship with mock-Lie admissible algebras. Some properties and double constructions of mock-pre-Lie algebras have been investigated in \cite{HH}. Furthermore, cohomologies and linear deformations of mock-pre-Lie algebras have been studied in \cite{Attan-Nabil}.

\subsection{Bialgebras}
For a given algebraic structure on a finite-dimensional vector space $\g$, to construct its corresponding bialgebra structure, one only needs to endow the dual space $\g^*$ with an algebraic structure of the same type, while imposing that the algebraic structures on $\g$ and $\g^*$ satisfy the corresponding compatibility conditions.
Beyond the direct construction of bialgebra structures, the significance of bialgebras also stems from the fact that the compatibility conditions admit two other equivalent formulations, namely, the equivalence relation among bialgebras, Manin triples and certain matched pairs. At present, many bialgebraic theories of algebra have been established and studied. V. G. Drinfel'd  introduced the notion of Lie bialgebras, which is closely related to Poisson-Lie groups \cite{K} and quantum universal enveloping algebras \cite{Drinfeld-cybe}. In \cite{Bai-chengming}, C. Bai developed the theory of pre-Lie bialgebras and established the connections with symplectic Lie algebras and phase spaces of Lie algebras.
The bialgebra theory of mock-Lie algebras has been investigated in \cite{BCHM,Cui-Hou,Ma-Sami-Makhlouf-Song}. The matched pair of mock pre-Lie algebras has also been studied in \cite{Na-Attan}. For further research on relevant bialgebra theory, refer to \cite{Benali,G3,Goncharov,Laraiedh-1,Laraiedh-2,Wang-Bai-Liu-Sheng}.

\subsection{Rota-Baxter operators and the classical Yang-Baxter equation}
Rota-Baxter operators on associative algebras originated in the study of fluctuation theory in probability theory. In recent years, Rota-Baxter operators have found wide-ranging applications, including Connes-Kreimer's work on the renormalization theory of quantum field theory \cite{CK}.
The theory of the Yang-Baxter equation was first discovered independently by the physicist C. N. Yang in 1967 \cite{Yang} and R. J. Baxter \cite{R.J.Baxter} in 1971, respectively. The classical Yang-Baxter equation can be regarded as the semi-classical limit of the quantum Yang-Baxter equation. Semenov-Tian-Shansky established the relationship between the classical Yang-Baxter equation and Rota-Baxter operators, showing that the operator form of the skew-symmetric classical Yang-Baxter equation is precisely corresponds to a Rota-Baxter operator of weight zero on a Lie algebra under certain conditions \cite{STS}.
B. A. Kupershmidt introduced the notion of $\huaO$-operators on Lie algebras \cite{Ku}, which are also called relative Rota-Baxter operators \cite{BBGN}. On the one hand, a skew-symmetric solution of the classical Yang-Baxter equation naturally gives rise to a relative Rota-Baxter operator with respect to the coadjoint representation \cite{Ku}. On the other hand, any relative Rota-Baxter operator can  give rise to a skew-symmetric solution of the classical Yang-Baxter equation \cite{Bai}. Rota-Baxter operators also play an important role in the theory of bialgebras \cite{Lang,RS}. For a more detailed introduction to Rota-Baxter operators, refer to \cite{Gub}.

\subsection{Main results and outline of the paper}
The purpose of this paper is to develop the theory of mock-pre-Lie bialgebras. First, we introduce the notion of
quadratic mock-pre-Lie algebras, and show that there is a one-to-one correspondence between symplectic mock-Lie algebras \cite{Baklouti-Benayadi} and quadratic mock-pre-Lie algebras. We study phase spaces of a mock-Lie algebra and prove that a mock-Lie algebra has a phase space if and only if it is sub-adjacent to a mock-pre-Lie algebra. We introduce the notion of a Manin triple of mock-pre-Lie algebras. Motivated by the general principle governing the equivalence between Manin triples and bialgebras, we further develop the theory of mock-pre-Lie bialgebras and establish the equivalence among Manin triples of mock-pre-Lie algebras, mock-pre-Lie bialgebras, the phase spaces of mock-Lie algebras, as well as the matched pair of mock-pre-Lie and mock-Lie algebras. We also introduce the notions of quasi-triangular and factorizable mock-pre-Lie bialgebras, showing that the former give rise to relative Rota-Baxter operators of weight $-1$. Finally, we use quadratic Rota-Baxter mock-pre-Lie algebras of weight $\lambda$ to characterize triangular and factorizable mock-pre-Lie bialgebras respectively.

The paper is organized as follows. In Section \ref{sec:I}, we introduce the notion of a phase space of a mock-Lie algebra, and characterize this structure in terms of mock-pre-Lie algebras and Manin triples of mock-pre-Lie algebras. In Section \ref{sec:II}, we introduce the notion of mock-pre-Lie bialgebras and establish their equivalence with Manin triples of mock-pre-Lie algebras, phase spaces of the sub-adjacent Lie algebras, as well as certain matched pairs of mock-pre-Lie and mock-Lie algebras. We also study coboundary mock-pre-Lie bialgebras, leading to an analogue of the classical Yang-Baxter equation. In Section \ref{sec:III}, we introduce the notion of quasi-triangular mock-pre-Lie bialgebras, which can give rise to a relative Rota-Baxter operator of weight $-1$. In Section \ref{sec:IV}, we focus on quadratic Rota-Baxter mock-pre-Lie algebras and show that factorizable and triangular mock-pre-Lie bialgebras
can be induced by such algebras.

Throughout this paper, all vector spaces are finite-dimensional over an algebraically closed field $\mathbb K$ of characteristic $0$. For any vector space $V$, $V^*=\Hom(V,\mathbb K)$ denotes its dual.
\vspace{2mm}

{\bf Acknowledgements. }  The first and last authors were supported by NSFC (12301034, 12461004, W2412041).
We express warmest thanks to Professor Yunhe Sheng for his valuable suggestions and  helpful discussions.

\section{The phase space of mock-Lie algebras and Manin triples of mock-pre-Lie algebras} \label{sec:I}
In this section, first we introduce the notion of  a phase space of a mock-Lie algebra. Then we prove that a mock-Lie algebra admits a phase space if and only if it is sub-adjacent to a mock-pre-Lie algebra. Moreover, we introduce the notion of a Manin triple of mock-pre-Lie algebras and show that there is a one-to-one correspondence between the phase spaces of mock-Lie algebras and Manin triples of mock-pre-Lie algebras.

To begin with, we recall some fundamental results concerning mock-Lie algebras and mock-pre-Lie algebras.

\begin{defi}\label{mock-Lie}{\rm(\cite{BBMM})}
A \textbf{mock-Lie algebra} is a vector space $\g$ together with a symmetric bilinear map $\{\cdot,\cdot\}_{\g}:\g \otimes \g \to \g$
such that
\begin{eqnarray}
    \{x,y\}_{\g}-\{y,x\}_{\g}&=&0 ,\\
    \{x,\{y,z\}_{\g}\}_{\g}+\{y,\{z,x\}_{\g}\}_{\g}+\{z,\{x,y\}_{\g}\}_{\g}&=&0,\quad \quad \forall x,y,z\in \g.
\end{eqnarray}
\end{defi}

\begin{defi}{\rm(\cite{Benayadi})}
A representation of a mock-Lie algebra $(\g, \{\cdot,\cdot\}_{\g})$ is a pair $(V;\rho),$
where $V$ is a vector space and $\rho : \g \rightarrow \gl(V)$ is a linear map such that
\begin{eqnarray}\label{rep-mockLie}
\rho(\{x,y\}_{\g}) = -\rho(x)\rho(y) - \rho(y)\rho(x),\quad \forall x,y \in \g.
\end{eqnarray}
\end{defi}

\begin{lem}\label{dual-rep-mockLie}
 Let $(V;\rho)$ be a representation of a mock-Lie algebra $(\g, \{\cdot,\cdot\}_{\g}).$  Define $\rho^*:\g\rightarrow\gl(V^*)$ by
 $$\langle\rho^*(x)\alpha,u\rangle=-\langle\alpha,\rho(x)u\rangle,\quad \forall x\in \g,\alpha\in V^*,u\in V.$$
 Then $(V^*;-\rho^*)$ is a representation of $(\g, \{\cdot,\cdot\}_{\g}),$ which is called the dual representation of $(V;\rho).$
\end{lem}

\begin{ex}\label{ad-coad}
 Let $(\g,\{\cdot,\cdot\}_{\g})$ be a mock-Lie algebra. Define $\ad:\g\rightarrow\gl(\g)$ by
 $$\ad_{x}y:=\{x,y\}_{\g},\quad \forall x,y\in \g.$$
 Then $(\g;\ad)$ is a representation of $\g$ on itself, which is called the adjoint representation. Furthermore, $(\g^*;-\ad^*)$ is also a representation of $\g,$ which is called the coadjoint representation.
\end{ex}

\begin{lem}{\rm(\cite{Benayadi})}
 Let $(\g,\{\cdot,\cdot\}_{\g})$ be a mock-Lie algebra. Then $(V;\rho)$ is a representation of
 $(\g,\{\cdot,\cdot\}_{\g})$ if and only if there is a mock-Lie algebra structure (called the semi-direct product) on the direct sum $\g\oplus V$ together with the multiplication defined by
 \begin{eqnarray}\label{semi-mock-lie}
   \{x+u,y+v\}_{\ltimes}:=\{x,y\}_{\g}+\rho(x)v+\rho(y)u,\quad\forall x,y\in\g,u,v\in V.
 \end{eqnarray}
  This semi-direct product mock-Lie algebra is denoted by $\g\ltimes_{\rho}V.$
\end{lem}

\begin{defi}{\rm(\cite{BCHM})}
Let $(\g, \{\cdot,\cdot\}_{\g})$ be a mock-Lie algebra and $(V;\rho)$ a representation. A linear map $T : V \rightarrow \g$ is called a relative Rota-Baxter operator of weight $0$ with respect to $(V;\rho)$
if $T$ satisfies
\begin{eqnarray*}
\{Tu,Tv\}_{\g} = T(\rho(Tu)v + \rho(Tv)u), \quad \forall u, v \in V.
\end{eqnarray*}
\end{defi}

The notion of mock-pre-Lie algebras (also called pre-Jacobi-Jordan algebras) was introduced in \cite{BBMM}.
\begin{defi}\label{mock-preLie}
A \textbf{mock-pre-Lie algebra} is a vector space $\g$ together with a bilinear multiplication $\cdot_{\g}: \g \otimes \g \to  \g$, such that for any $x,y,z\in \g$, the anti-associator $Aass(x,y,z)=(x\cdot_{\g} y)\cdot_{\g} z+x\cdot_{\g} (y\cdot_{\g} z)$ is skew-symmetric in $x,y,$ i.e. $Aass(x,y,z)=-Aass(y,x,z)$ or equivalently,
\begin{eqnarray}\label{defi-mock-preLie-equation}
(x\cdot_{\g} y)\cdot_{\g} z+x\cdot_{\g} (y\cdot_{\g} z)=-(y\cdot_{\g} x)\cdot_{\g} z-y\cdot_{\g}(x\cdot_{\g} z).
\end{eqnarray}
\end{defi}

\begin{pro}
\label{pre-jj to jj}{\rm(\cite{BBMM})} Let $(\g, \cdot_{\g})$ be a mock-pre-Lie algebra.  Then the operation
\begin{equation}\label{commutator}
\{x,y\}_{c}:=x\cdot_{\g} y + y\cdot_{\g} x,\quad\forall x,y\in \g
\end{equation} defines a mock-Lie algebra $(\g,\{\cdot,\cdot\}_{c})$, which is called the {\bf sub-adjacent mock-Lie algebra} of $(\g, \cdot_{\g})$ and denoted by $\g_c.$  $(\g, \cdot_{\g})$ is called a {\bf compatible mock-pre-Lie algebra}
on the mock-Lie algebra  $\g_c$. Furthermore, for $x\in \g,$ the linear map  $L:\g\rightarrow\gl(\g)$ defined by $L_{x}y = x\cdot_{\g} y$ gives a representation of the mock-Lie algebra $\g_c$ on $\g.$
\end{pro}

\begin{ex}\label{direct-sum-mock-pre}
Let $(\g,\cdot_{\g})$ and $(\h,\cdot_{\h})$ be two mock-pre-Lie algebras. Then $(\g\oplus\h,\cdot_{\oplus})$ is a mock-pre-Lie algebra,
  where the multiplication $\cdot_{\oplus}$ is defined by
\begin{eqnarray*}
 (x_1,y_1)\cdot_{\oplus}(x_2, y_2):=(x_1\cdot_{\g} x_2,y_1\cdot_{\h} y_2),\quad \forall x_1,x_2\in \g,y_1,y_2\in \h.
\end{eqnarray*}
This mock-pre-Lie algebra is  called the direct sum mock-pre-Lie algebra.
\end{ex}

\begin{defi}{\rm(\cite{Attan-Nabil})}
A {\bf representation} of a mock-pre-Lie algebra $(\g,\cdot_{\g})$ is a triple $(V;\rho,\mu),$ where $V$ is a vector space, $\rho:\g\rightarrow \gl(V)$ is a representation of the sub-adjacent mock-Lie algebra $\g_c$ on $V$ and $\mu: \g\rightarrow \gl(V)$ is a linear map satisfying
\begin{eqnarray}\label{mock-pre-rep}
\rho(x)\mu(y)+\mu(y)\rho(x)=-\mu(x\cdot_{\g} y)-\mu(y)\mu(x),\quad \forall x,y \in \g.
\end{eqnarray}
\end{defi}

The following semidirect product characterization of representation was given in \cite{Attan-Nabil}.

\begin{lem}{\rm(\cite{Attan-Nabil})}
Let $(\g,\cdot)$ be a mock-pre-Lie algebra, $V$ is a vector space and $\rho,\mu:\g\rightarrow\gl(V)$ are linear maps. Then $(V;\rho,\mu)$ is a representation of a mock-pre-Lie algebra $(\g,\cdot_{\g})$ if
and only if there is a mock-pre-Lie algebra structure on the direct sum of vector spaces $\g\oplus V$ with the multiplication  defined by
\begin{equation}
(x+u)\cdot_{(\rho,\mu)}(y+v):=x\cdot_{\g} y+\rho(x)v+\mu(y)u,\quad \forall x,y\in \g, u,v\in V.
\end{equation} This semidirect product mock-pre-Lie algebra is denoted by $\g\ltimes_{\rho,\mu}V$.
\end{lem}

\begin{ex}\label{regular-representation-mock-pre-Lie}
Let $(\g,\cdot_{\g})$ be a mock-pre-Lie algebra. Then $(\g; L, R)$ is a representation of $(\g,\cdot_{\g}),$ which is called the {\bf regular representation}, where  $R:\g\rightarrow\gl(\g)$ is the right multiplication defined by $R_{x}y=y\cdot_{\g} x,$ for all $x, y\in \g$. If there is a mock-Lie algebra structure on the dual space $\g^*,$ we denote the left and right multiplications by $\huaL$ and $\huaR$ respectively.
\end{ex}

Next we consider the dual representation of a mock-pre-Lie algebra. Let $(V;\rho,\mu)$ be a representation of a mock-pre-Lie algebra $(\g,\cdot_{\g})$. For all $x\in\g,u\in V,\xi\in V^*,$ define
$\rho^*:\g\rightarrow\gl(V^*)$ and $\mu^*:\g\rightarrow\gl(V^*)$ by
\begin{eqnarray*}
  \langle \rho^{*}(x)\xi, u\rangle = -\langle \xi, \rho(x)u\rangle, \quad
  \langle \mu^*(x)\xi, u\rangle = -\langle \xi, \mu(x)u\rangle.
\end{eqnarray*}
\begin{pro}\label{dual-rep-mockpre}
Let $(V;\rho,\mu)$ be a representation of a mock-pre-Lie algebra $(\g,\cdot_{\g})$. Then
$$(V^*;-\rho^*-\mu^*,\mu^*)$$
is a representation of $(\g,\cdot_{\g})$, which is called the {\bf dual representation} of $(V;\rho,\mu).$
\end{pro}
\begin{proof}
For all $x,y\in \g,u\in V,\xi\in V^*,$ by \eqref{rep-mockLie} and \eqref{mock-pre-rep}, we have
\begin{eqnarray*}
&&\langle(-\rho^*-\mu^*)(\{x,y\}_{c})\xi+(-\rho^*-\mu^*)(x)(-\rho^*-\mu^*)(y)\xi+(-\rho^*-\mu^*)(y)(-\rho^*-\mu^*)(x)\xi,u\rangle\\
&=&\langle \xi, (\rho(\{x,y\}_{c})+\mu(\{x,y\}_{c})+\rho(y)\rho(x)+\rho(y)\mu(x)+\mu(y)\rho(x)\\
&&+\mu(y)\mu(x)+\rho(x)\rho(y)+\rho(x)\mu(y)+\mu(x)\rho(y)+\mu(x)\mu(y))u\rangle\\
&=&0,
\end{eqnarray*}
which implies that $-\rho^*-\mu^*$ is a representation of the sub-adjacent mock-Lie algebra $\g_{c}.$
Moreover,
\begin{eqnarray*}
&&\langle (-\rho^*-\mu^*)(x)\mu^*(y)\xi+\mu^*(y)(-\rho^*-\mu^*)(x)\xi+\mu^*(x\cdot_{\g} y)\xi+\mu^*(y)\mu^*(x)\xi,u\rangle\\
&=&\langle \xi, (\mu(x)\mu(y)-\mu(y)(\rho(x)+\mu(x))-(\rho(x)+\mu(x))\mu(y)-\mu(x\cdot_{\g} y))u \rangle\\
&=&0,
\end{eqnarray*}
implies that \eqref{mock-pre-rep} holds.
Therefore, $(V^*;-\rho^*-\mu^*,\mu^*)$ is a representation of $(\g,\cdot_{\g})$.
\end{proof}

\begin{ex}
{\rm Let $(\g,\cdot_{\g})$ be a mock-pre-Lie algebra. By Proposition \ref{dual-rep-mockpre}, $(\g^*;-L^*-R^*,R^*)$ is a representation of $(\g,\cdot_{\g})$, which is called the {\bf coregular representation}, where two linear maps $L^*,R^*:\g\rightarrow\gl(\g^*)$ are respectively defined by
\begin{eqnarray*}
  \langle L^*_x(\xi),y\rangle=-\langle\xi,x\cdot_{\g} y\rangle,\quad \langle R^*_x(\xi),y \rangle=-\langle \xi,y\cdot_{\g} x\rangle,\quad \forall x,y\in \g,\xi\in \g^*.
\end{eqnarray*}}
\end{ex}

\begin{defi}{\rm(\cite{Baklouti-Benayadi})}
A \textbf{symplectic structure} on a mock-Lie algebra $(\g,\{\cdot,\cdot\}_{\g})$ is a non-degenerate
skew-symmetric bilinear form $\omega\in \wedge^2 \g^*$ satisfying
\begin{eqnarray}\label{symplectic-mock-Lie}
\omega(\{x,y\}_{\g}, z) + \omega(\{y,z\}_{\g}, x) + \omega(\{z,x\}_{\g}, y) = 0, \quad \forall x, y, z \in \g.
\end{eqnarray}
\end{defi}
A mock-Lie algebra $\g$ is called \textbf{symplectic mock-Lie algebra} if it is endowed with a symplectic
form $\omega$, and it is denoted by $(\g,\omega)$.

\begin{pro}{\rm(\cite{Baklouti-Benayadi}\label{sym-mock-pre})}
Let $(\g,\omega)$ be a symplectic mock-Lie algebra. Then there exists a compatible mock-pre-Lie algebra structure $(\g,\cdot_{\g})$ given by
\begin{eqnarray}\label{omega-Mpre}
\omega(x \cdot_{\g} y, z) = \omega(y, \{x,z\}_{\g}), \quad
\forall x, y, z \in \g.
\end{eqnarray}
\end{pro}

Next we introduce the concept of quadratic mock-pre-Lie algebras, which are not only closely related to symplectic mock-Lie algebras, but also serve as the natural background for Manin triples of mock-pre-Lie algebras.
\begin{defi}
A {\bf quadratic mock-pre-Lie algebra} $(\g,\cdot_{\g},\omega)$ is a mock-pre-Lie algebra $(\g,\cdot_{\g})$ equipped with a nondegenerate skew-symmetric bilinear form $\omega\in \wedge^2 \g^*$ such that the following invariant condition holds:
\begin{eqnarray}\label{quadratic-mock-pre}
\omega(x \cdot_{\g} y, z) = \omega(y, \{x,z\}_{c}), \quad \forall x, y, z \in \g.
\end{eqnarray}
\end{defi}
\begin{rmk}
Actually, by \eqref{quadratic-mock-pre}, we also have
\begin{eqnarray}\label{quadratic-mock-pre-2}
\omega(x \cdot_{\g} y, z) =\omega(z\cdot_{\g}y,x).
\end{eqnarray}
\end{rmk}
Proposition \ref{sym-mock-pre} tells us that quadratic mock-pre-Lie algebras are the underlying structures of symplectic mock-Lie algebras.

\begin{pro}\label{symplectic-quadratic}
There is a one-to-one correspondence between symplectic mock-Lie algebras and quadratic mock-pre-Lie algebras.
\end{pro}
\begin{proof}
Let  $(\g,\cdot_{\g},\omega)$ be a quadratic mock-pre-Lie algebra. By \eqref{commutator} and
\eqref{quadratic-mock-pre}, we have
\begin{eqnarray}
\omega(\{x,y\}_{c}, z) + \omega(\{y,z\}_{c}, x) + \omega(\{z,x\}_{c}, y) = 0, \quad \forall x, y, z \in \g.
\end{eqnarray}
Thus, $(\g,\{\cdot,\cdot\}_{c},\omega)$ is a symplectic mock-Lie algebra.

Conversely, let $(\g,\{\cdot,\cdot\}_{\g},\omega)$ be a symplectic mock-Lie algebra. By Proposition \ref{sym-mock-pre}, there is a compatible quadratic mock-pre-Lie algebra structure $(\g,\cdot_{\g},\omega)$ given by
\begin{eqnarray*}
\omega(x \cdot_{\g} y, z) = \omega(y, \{x,z\}_{\g}).
\end{eqnarray*}
The proof is finished.
\end{proof}

Let $V$ be a vector space and $V^*$ its dual space. Then there is a natural nondegenerate skew-symmetric bilinear form $\omega$ on $V\oplus V^*$ given by:
\begin{eqnarray}\label{omega}
\omega(x+\xi,y+\eta)=\langle\xi,y\rangle-\langle\eta,x\rangle,\quad \forall x,y\in \g, \xi,\eta\in \g^*.
\end{eqnarray}

Next we introduce the notion of a phase space of a mock-Lie algebra.
\begin{defi}
Let $(\g,\{\cdot,\cdot\}_{\g})$ be a mock-Lie algebra and let $\g^*$ be  dual space of $\g$.
If there is a mock-Lie algebra structure on the direct sum  $\g\oplus \g^*$ such that $(\g\oplus \g^*,\{\cdot,\cdot\},\omega)$ is a symplectic mock-Lie algebra, where $\omega$ is given by \eqref{omega}, and $(\g,\{\cdot,\cdot\}_{\g})$ and $(\g^*,\{\cdot,\cdot\}|_{\g^*})$ are mock-Lie subalgebras of $(\g\oplus \g^*,\{\cdot,\cdot\})$, then the symplectic mock-Lie algebra $(\g\oplus \g^*,\{\cdot,\cdot\},\omega)$ is called a {\bf phase space of the mock-Lie algebra $(\g,\{\cdot,\cdot\}_{\g})$}.
\end{defi}

Mock-pre-Lie algebras play an important role in the study of phase spaces of mock-Lie algebras.

\begin{thm}\label{phasespace-subadj}
A mock-Lie algebra has a phase space if and only if it is sub-adjacent to a mock-pre-Lie algebra.
\end{thm}

\begin{proof}
Let $(\g,\cdot)$ be a mock-pre-Lie algebra. By Proposition \ref{pre-jj to jj}, the left multiplication $L$ is a representation of the mock-Lie algebra $\g_c$ on $\g$. By Lemma \ref{dual-rep-mockLie}, $-L^*$ is a representation of the sub-adjacent mock-Lie algebra $\g_c$ on $\g^*$. Thus, we have the semidirect product mock-Lie algebra $\g_c\ltimes_{-L^*} \g^*=(\g_c\oplus \g^*,\{\cdot,\cdot\}_{\ltimes})$. By \eqref{omega}, for all $x,y,z\in \g$ and $\xi,\eta,\zeta\in \g^*,$ we have
\begin{eqnarray*}
&&\omega(\{x+\xi,y+\eta\}_{\ltimes},z+\zeta)+\omega(\{y+\eta,z+\zeta\}_{\ltimes},x+\xi)+\omega(\{z+\zeta,x+\xi\}_{\ltimes},y+\eta)\\
&=&\omega(\{x,y\}-L^*_{x}\eta-L^*_{y}\xi,z+\zeta)+\omega(\{y,z\}-L^*_{y}\zeta-L^*_{z}\eta,x+\xi)+\omega(\{z,x\}-L^*_{z}\xi-L^*_{x}\zeta,y+\eta)\\
&=&\langle-L^*_{x}\eta-L^*_{y}\xi,z\rangle-\langle\zeta,\{x,y\}_{c}\rangle
+\langle-L^*_{y}\zeta-L^*_{z}\eta,x\rangle-\langle\xi,\{y,z\}_{c}\rangle
+\langle-L^*_{z}\xi-L^*_{x}\zeta,y\rangle-\langle\eta,\{z,x\}_{c}\rangle\\
&=&\langle\eta,L_{x}z\rangle+\langle\xi,L_{y}z\rangle-\langle\zeta,\{x,y\}_{c}\rangle
+\langle\zeta,L_{y}x\rangle+\langle\eta,L_{z}x\rangle\\
&&-\langle\xi,\{y,z\}_{c}\rangle
+\langle\xi,L_{x}y\rangle+\langle\zeta,L_{x}y\rangle-\langle\eta,\{z,x\}_{c}\rangle\\
&=&0,
\end{eqnarray*}
which implies that $(\g_c\ltimes_{-L^*} \g^*,\omega)$ is a symplectic mock-Lie algebra. Moreover, $(\g_c,\{\cdot,\cdot\}_{c})$ is a subalgebra of $\g_c\ltimes_{-L^*} \g^*$ and $\g^*$ is an abelian subalgebra of $\g_c\ltimes_{-L^*} \g^*$. Thus, the symplectic mock-Lie algebra $(\g_c\ltimes_{-L^*} \g^*,\omega)$ is a phase space of the sub-adjacent mock-Lie algebra $(\g_c,\{\cdot,\cdot\}_{c})$.

Conversely, let $(\g\oplus \g^*,\{\cdot,\cdot\},\omega)$  be a phase space of a mock-Lie algebra $(\g,\{\cdot,\cdot\}_{\g})$. By Proposition \ref{sym-mock-pre}, there exists a compatible mock-pre-Lie algebra structure $\cdot$ on $\g\oplus \g^*$ given by
\begin{eqnarray*}
  \omega((x+\xi)\cdot (y+\eta),z+\zeta)= \omega(y+\eta,\{x+\xi,z+\zeta\}).
\end{eqnarray*}
 Since $(\g,\{
\cdot,\cdot\}_{\g})$ is a subalgebra of $(\g\oplus \g^*,\{\cdot,\cdot\})$, we have
\begin{eqnarray*}
\omega(x \cdot y, z) = \omega(y, \{x,z\})=\omega(y, \{x,z\}_{\g})=\langle y,0\rangle-\langle 0,\{x,z\}_{\g}\rangle=0,\quad \forall x,y,z\in \g.
\end{eqnarray*}
Thus, $x\cdot y\in \g,$ which implies that $(\g,\cdot|_\g)$ is a subalgebra of the mock-pre-Lie algebra $(\g\oplus \g^*,\cdot)$. Its sub-adjacent mock-Lie algebra $(\g_c,\{\cdot,\cdot\}_c)$ is exactly the original mock-Lie algebra $(\g,\{\cdot,\cdot\}_{\g})$.
\end{proof}

\begin{cor}\label{phase-space-submock-pre}
Let $(\g\oplus \g^*,\{\cdot,\cdot\},\omega)$  be a phase space of a mock-Lie algebra and $(\g\oplus \g^*,\cdot)$ the associated mock-pre-Lie algebra. Then both $(\g,\cdot|_{\g})$ and $(\g^*,\cdot|_{\g^*})$ are subalgebras of the mock-pre-Lie algebra $(\g\oplus \g^*,\cdot)$.
\end{cor}

\begin{cor}
Let $(V;\rho)$ be a representation of the mock-Lie algebra $(\g,\{\cdot,\cdot\})$ and  $(V^*;-\rho^*)$
the dual representation. If $(\g\oplus \g^*,\{\cdot,\cdot\}_{\ltimes},\omega)$ is a phase space of $(\g,\{\cdot,\cdot\}),$ where $\{\cdot,\cdot\}_{\ltimes}$ is given by \eqref{semi-mock-lie}, then
\begin{eqnarray*}
x\cdot y:= \rho(x)y,\quad \forall x,y\in \g,
\end{eqnarray*}
defines a mock-pre-Lie algebra structure on $\g.$
\end{cor}

\begin{proof}
For all $x,y,z\in \g$ and $\xi\in \g^*$, by \eqref{quadratic-mock-pre} and \eqref{omega}, we have
\begin{eqnarray*}
\langle\xi,x\cdot y\rangle=\omega(\xi,x\cdot y)=-\omega(x\cdot y,\xi)=-\omega(y,\{x,\xi\}_{\ltimes})=\omega(y,\rho^{*}(x)\xi)=-\langle \rho^{*}(x)\xi,y\rangle=\langle \xi,\rho(x)y\rangle.
\end{eqnarray*}
Therefore, we have $x\cdot y=\rho(x)y.$
\end{proof}

At the end of this section, we introduce the notion of a Manin triple of mock-pre-Lie algebras.

\begin{defi}
A {\bf Manin triple of mock-pre-Lie algebras} is a triple $(\huaG,\g,\g')$, where
\begin{itemize}
\item[{\rm(i)}] $(\huaG,\cdot,\omega)$ is a quadratic mock-pre-Lie algebra;

\item[{\rm(ii)}] both $\g$ and $\g'$ are isotropic subalgebras of $(\huaG,\cdot)$, that is, $\omega(x,y)=\omega(u,v)=0$ for $x,y\in \g$ and $u,v\in \g'$;

\item[{\rm(iii)}]$\huaG=\g\oplus \g'$ as vector spaces.
\end{itemize}
\end{defi}

In a Manin triple of mock-pre-Lie algebras $(\huaG,\g,\g')$, since the skew-symmetric bilinear form $\omega$ is nondegenerate, $\g'$ can be identified with $\g^*$ via
$$\langle \xi,x \rangle:=\omega(\xi,x),\quad \forall x\in \g,\xi\in \g'.$$
Thus, $\huaG$ is isomorphic to $\g\oplus \g^*$ naturally and the bilinear form $\omega$ is exactly given by
\eqref{omega}.

\begin{thm}\label{phase-space-Manin-triple}
There is a one-to-one correspondence between Manin triples of mock-pre-Lie algebras and phase spaces of mock-Lie algebras. More precisely, if $(\g\oplus \g^*,\g,\g^*)$ is a Manin triple of mock-pre-Lie algebras, then $(\g\oplus \g^*,\{\cdot,\cdot\},\omega)$ is a phase space of the mock-Lie algebra $(\g,\{\cdot,\cdot\}_{\g}),$ where $\omega$ is given by \eqref{omega}. Conversely, if $(\g\oplus \g^*,\{\cdot,\cdot\}_{\g\oplus \g^*},\omega)$ is a phase space of a mock-Lie algebra $(\g,\{\cdot,\cdot\}_{\g}),$ then $(\g\oplus \g^*,\g,\g^*)$ is a Manin-triple of mock-pre-Lie algebras, where the mock-pre-Lie algebra structure on $\g\oplus \g^*$ is given by \eqref{quadratic-mock-pre}.
\end{thm}

\begin{proof}
Let  $(\g\oplus \g^*,\g,\g^*)$ be a Manin triple of mock-pre-Lie algebras. Here $(\g\oplus \g^*,\cdot,\omega)$ is a quadratic mock-pre-Lie algebra and $\omega$ is given by \eqref{omega}.
Denote by $\cdot_{\g}$ and $\cdot_{\g^{*}}$ be the mock-pre-Lie algebra structures on $\g$ and $\g^*$, respectively, and denote by $\{\cdot,\cdot\}_{\g}$ and $\{\cdot,\cdot\}_{\g^*}$ the corresponding sub-adjacent mock-Lie algebra structures on $\g$ and $\g^*$, respectively. It is straightforward to deduce that the corresponding mock-Lie algebra structure $\{\cdot,\cdot\}$ on $\g\oplus \g^*$ is given by
\begin{eqnarray}\label{matched-pair-MockLie}
\{x+\xi,y+\eta\}=\{x,y\}_{\g}-\huaL^*_{\xi}y-\huaL^*_{\eta}x+\{\xi,\eta\}_{\g^*}-L^*_{x}\eta-L^*_{y}\xi, \quad \forall x,y\in \g,\xi,\eta\in \g^*,
\end{eqnarray}
it is obvious that $(\g,\{\cdot,\cdot\}_{\g})$ and $(\g^*,\{\cdot,\cdot\}_{\g^*})$ are subalgebras of $(\g\oplus \g^*,\{\cdot,\cdot\})$.
By Proposition \ref{symplectic-quadratic}, $(\g\oplus \g^*,\{\cdot,\cdot\},\omega)$ is a symplectic mock-Lie algebra, which is a phase space of the mock-Lie algebra $(\g,\{\cdot,\cdot\}_{\g}).$

Conversely, let $(\g\oplus \g^*,\{\cdot,\cdot\}_{\g\oplus \g^*},\omega)$ be a phase space of a mock-Lie algebra $(\g,\{\cdot,\cdot\}_\g)$. By Proposition \ref{sym-mock-pre}, there exists a mock-pre-Lie algebra structure $\cdot$ on $\g\oplus \g^*$ given by \eqref{omega-Mpre} such that $(\g\oplus \g^*,\cdot,\omega)$ is a quadratic mock-pre-Lie algebra. By Corollary \ref{phase-space-submock-pre},
$(\g,\cdot|_{\g})$ and $(\g^*,\cdot|_{\g^*})$ are subalgebras of the mock-pre-Lie algebra $(\g\oplus \g^*,\cdot)$. It is obvious that both $\g$ and $\g^*$ are isotropic. Therefore, $(\g\oplus \g^*,\g,\g^*)$ is a Manin triple of mock-pre-Lie algebras.

\end{proof}

\section{Mock-pre-Lie bialgebras and the mock-pre-Lie classical Yang-Baxter equation}\label{sec:II}
In this section, first we introduce the notion of mock-pre-Lie bialgebras. Then we establish the equivalence between mock-pre-Lie bialgebras, Manin triples of mock-pre-Lie algebras, phase spaces of the sub-adjacent Lie algebras, certain matched pairs of mock-pre-Lie algebras and certain matched pairs of mock-Lie algebras. Furthermore, we introduce the notion of coboundary mock-pre-Lie bialgebras, which yields the mock-pre-Lie classical Yang-Baxter equation, as an analogue of the classical Yang-Baxter equation.

\subsection{Mock-pre-Lie bialgebras}
For a mock-pre-Lie algebra $(\g^*,\cdot_{\g^*}),$ let $\alpha:\g\rightarrow\otimes^2{\g}$ be the dual map of
$\cdot_{\g^*}:\otimes^2\g^*\rightarrow\g^*,$ i.e.
\begin{eqnarray*}
  \langle \alpha(x),\xi\otimes\eta\rangle=\langle x,\xi\cdot_{\g^*}\eta\rangle
\end{eqnarray*}

\begin{defi}
Let $\g$ be a vector space. A {\bf mock-pre-Lie bialgebra} structure on $\g$ is a pair of linear maps $(\alpha,\beta)$ where $\alpha:\g\rightarrow\g\otimes \g,$
$\beta:\g^*\rightarrow\g^*\otimes\g^*,$ such that
\begin{itemize}
\item[{\rm(i)}] $\alpha^*:\g^*\otimes\g^*\rightarrow\g^*$ is a mock-pre-Lie algebra structure on $\g^*;$
\item[{\rm(ii)}] $\beta^*:\g\otimes\g\rightarrow\g$ is a mock-pre-Lie algebra structure on $\g;$
\item[{\rm(iii)}] For all $x,y\in \g,\xi,\eta \in \g^*,$ $\alpha$ and $\beta$ satisfied the following compatibility conditions:
\begin{align}
\label{alpha-condition}\alpha\{x,y\}_{\g}&=-({\Id}\otimes (R_{y}+L_{y})+L_{y}\otimes {\Id})\alpha(x)-({\Id}\otimes (L_{x}+R_{x})+L_{x}\otimes {\Id})\alpha(y),\\
\label{beta-condition} \beta\{\xi,\eta\}_{\g^*}&=-({\Id}\otimes (\huaR_{\eta}+\huaL_{\eta})+\huaL_{\eta}\otimes {\Id})\beta(\xi)-({\Id}\otimes (\huaL_{\xi}+\huaR_{\xi})+\huaL_{\xi}\otimes {\Id})\beta(\eta).
\end{align}
\end{itemize}
We denote this {\bf mock-pre-Lie bialgebra} by $(\g,\g^*,\alpha,\beta).$ 
\end{defi}

Next we recall the notion of a matched pair of mock-Lie algebras and a matched pair of mock-pre-Lie algebras.

\begin{defi}{\rm(\cite{Na-Attan})}
Let $(\g, \cdot_{\g})$ and $(\h, \cdot_{\h})$ be two mock-pre-Lie algebras. Let $(\g,\{\cdot,\cdot\}_{\g})$ and $(\h,\{\cdot,\cdot\}_{\h})$ be the sub-adjacent mock-Lie algebras of $(\g, \cdot_{\g})$ and $(\h, \cdot_{\h})$ respectively. If there exists a representation $(\rho_{\g},\mu_{\g})$ of $\g$ on $\h$ and a representation $(\rho_{\h},\mu_{\h})$ of
$\h$ on $\g$ such that the following compatibility conditions are satisfied:
\begin{align}
\label{mockpair1}0=&\mu_{\h}(u)(\{x,y\}_{\g})+\mu_{\h}(\rho_{\g}(y)u)x + \mu_{\h}(\rho_{\g}(x)u)y + x \cdot_{\g} (\mu_{\h}(u)y) + y \cdot_{\g} (\mu_{\h}(u)x),\\ \label{mockpair2}0=&\mu_{\g}(x)(\{u,v\}_{\h})+\mu_{\g}(\rho_{\h}(v)x)u + \mu_{\g}(\rho_{\h}(u)x)v + u \cdot_{\h} (\mu_{\g}(x)v) + v \cdot_{\h} (\mu_{\g}(x)u),\\
\label{mockpair3}-\rho_{\h}(u)(x \cdot_{\g} y)=&\rho_{\h}(\rho_{\g}(x)u + \mu_{\g}(x)u)y + (\rho_{\h}(u)x + \mu_{\h}(u)x) \cdot_{\g} y + \mu_{\h}(\mu_{\g}(y)u)x + x \cdot_{\g} (\rho_{\h}(u)y),\\
\label{mockpair4}-\rho_{\g}(x)(u \cdot_{\h} v)=&\rho_{\g}(\rho_{\h}(u)x + \mu_{\h}(u)x)v + (\rho_{\g}(x)u + \mu_{\g}(x)u) \cdot_{\h} v + \mu_{\g}(\mu_{\h}(v)x)u + u \cdot_{\h} (\rho_{\g}(x)v),
\end{align}
for all $x, y \in \g, u,v \in \h,$ then we call $(\g, \h; (\rho_{\g}, \mu_{\g}), (\rho_{\h}, \mu_{\h}))$ a {\bf matched pair of mock-pre-Lie algebras}.
\end{defi}
In fact, we have the following alternative description of matched pairs of mock-pre-Lie algebras.
\begin{pro}\label{pro-matched-pair-mock-pre}
Let $(\g, \h; (\rho_{\g}, \mu_{\g}), (\rho_{\h}, \mu_{\h}))$ be a matched pair of mock-pre-Lie algebras. Then there is a mock-pre-Lie algebra structure $\cdot_{\bowtie}:\otimes^2(\g \oplus \h)\rightarrow \g \oplus \h$ defined by
\begin{equation}\label{matched-pair-mock-pre}
(x+u)\cdot_{\bowtie}(y+v)=x \cdot_{\g} y + \rho_{\h}(u)y+ \mu_{\h}(v)x  +u\cdot_{\h}v + \rho_{\g}(x)v + \mu_{\g}(y)u, \quad \forall x, y \in \g, u, v \in \h.
\end{equation}

Conversely, if $(\g\oplus\h, \cdot_{\bowtie})$ is a mock-pre-Lie algebra such that $\g$ and $\h$ are mock-pre-Lie subalgebras, then $(\g, \h; (\rho_{\g}, \mu_{\g}), (\rho_{\h}, \mu_{\h}))$ is a matched pair of mock-pre-Lie algebras, where the representation $(\rho_{\g}, \mu_{\g})$ of $\g$ on $\h$ and the representation $(\rho_{\h}, \mu_{\h})$ of $\h$ on $\g$ are determined by
\begin{eqnarray}\label{mock-pre-Lie-rho-mu}
x\cdot_{\bowtie}u=\rho_{\g}(x)u+\mu_{\h}(u)x,\quad u\cdot_{\bowtie}x=\mu_{\g}(x)u+\rho_{\h}(u)x,\quad \forall x\in \g,u\in \h.
\end{eqnarray}
\end{pro}

\begin{defi} {\rm(\cite{Na-Attan})}
Let $(\g,\{\cdot,\cdot\}_{\g})$ and $(\h,\{\cdot,\cdot\}_{\h})$ be two mock-Lie algebras. If there exists a representation $\rho:\g\rightarrow\g(\h)$ of $\g$
on $\h$ a representation $\mu:\h\rightarrow \gl(\g)$ of $\h$ on $\g$ such that
\begin{eqnarray}
 \label{matched-ML1}\rho(x)(\{u,v\}_{\h})+\{\rho(x)u, v\}_{\h} +\{u,  \rho(x)v\}_{\h}+\rho(\mu(u)x)v+\rho(\mu(v)x)u&=&0,\\
    \label{matched-ML2}\mu(u)(\{x, y\}_{\g})+\{\mu(u)x , y\}_{\g} +\{x, \mu(u)y\}_{\g}+\mu(\rho(x)u)y+\mu(\rho(y)u)x&=&0,
\end{eqnarray}
for all $x,y\in \g,u,v\in \h,$ then we call $(\g, \h; \rho, \mu)$ a {\bf matched pair of mock-Lie algebras.} We will denote a matched pair of mock-Lie algebras by $(\g, \h; \rho, \mu)$, or simply by $(\g, \h)$.
\end{defi}

In this case, there is a mock-Lie algebra
structure on the direct sum  $\g\oplus \h$ with the multiplication defined by
\begin{equation}\label{eqmatch}
    \{x+u,y+v\}_{\bowtie}:=\{x,y\}_{\g} +\mu(v)x+\mu(u)y+\{u,v\}_{\h}+\rho(y)u+\rho(x)v,\;\forall x,y\in \g,\;u,v\in \h.
\end{equation}
This mock-Lie algebra is denoted by $\g {\bowtie} \h$.
Conversely, if $(\g\oplus \h,\{\cdot,\cdot\})$ is a mock-Lie algebra such that $\g$ and $\h$ are mock-Lie subalgebras, then $(\g, \h; \rho,\mu)$ is a matched pair of mock-Lie algebras, where the representations $\rho:\g\rightarrow\g(\h)$ and $\mu:\h\rightarrow \gl(\g)$  are determined by
\begin{eqnarray}
\{x,u\}=\rho(x)u+\mu(u)x, \quad \forall x\in \g, u\in\h.
\end{eqnarray}

\begin{pro}{\rm(\cite{Na-Attan})}\label{ref-mocklie-mockpre}
Let $(\g,\h;(\rho_{\g},\mu_{\g}),(\rho_{\h},\mu_{\h}))$ be a matched pair of mock-pre-Lie algebras. Then $(\g_{c},\h_{c};\rho_{\g}+\mu_{\g},\rho_{\h}+\mu_{\h})$ is a matched pair of mock-Lie algebras.
\end{pro}

\begin{thm}\label{equivalent-match pre-Lie}
Let $(\g,\cdot_{\g})$ and $(\g^*,\cdot_{\g^*})$ be two mock-pre-Lie algebras. Then $(\g,\g^*;(-L^*-R^*, R^*),(-\huaL^*-\huaR^*, \huaR^*))$ is a matched pair of mock-pre-Lie algebras if and only if
$(\g_{c},\g_{c}^*;-L^*,-\huaL^*)$ is a matched pair of mock-Lie algebras.
\end{thm}
\begin{proof}
By Proposition \ref{ref-mocklie-mockpre}, we can obtain that if $(\g,\g^*;(-L^*-R^*, R^*),(-\huaL^*-\huaR^*, \huaR^*))$ is a matched pair of mock-pre-Lie algebras, then $(\g_{c},\g_{c}^*;-L^*,-\huaL^*)$ is a matched pair of mock-Lie algebras.  Conversely, we give a direct proof as follows. Let $\rho_{\g}=-L^*-R^*,\mu_{\g}=R^*,\rho_{\h}=-\huaL^*-\huaR^*,\mu_{\h}=\huaR^*$ and $\rho=-L^*,\mu=-\huaL^*,$ we have
\begin{eqnarray*}
{\rm equation~\eqref{mockpair2}}\Longleftrightarrow {\rm equation~\eqref{mockpair3}}&\Longleftrightarrow&{\rm equation~\eqref{matched-ML1}},\\
{\rm equation~\eqref{mockpair1}}\Longleftrightarrow {\rm equation~\eqref{mockpair4}}&\Longleftrightarrow&{\rm equation~\eqref{matched-ML2}}.
\end{eqnarray*}
We show an example how the equation \eqref{mockpair2} is equivalent to the equation \eqref{matched-ML1}. For all $x,y\in \g,u,v\in \h,$ we have
\begin{eqnarray*}
\langle R_x^*\{u,v\}_{\h},y\rangle&=&\langle L_y^*\{u,v\}_{\h},x\rangle;\\
\langle R^* \big(-\huaL_v^*(x)-\huaR_v^*(x) \big)u,y\rangle&=&\langle \{L_y^* u,v\}_{\h},x\rangle;\\
\langle R^* \big(-\huaL_u^*(x)-\huaR_u^*(x) \big)v,y\rangle&=&\langle \{L_y^* v,u\}_{\h},x\rangle;\\
\langle u\cdot_{\h}R_{x}^*v,y\rangle&=&-\langle L^*(\huaL^*_{u}y)v,x\rangle;\\
\langle v\cdot_{\h}R_{x}^*u,y\rangle&=&-\langle  L^*(\huaL^*_{v}y)u,x\rangle.
\end{eqnarray*}
The equivalence relations among the other equations can be proved similarly, and we omit the details here.
\end{proof}

\begin{thm}
Let $(\g,\cdot_{\g})$ and  $(\g^*,\cdot_{\g^*})$ be two mock-pre-Lie
algebras. Then the following conditions are equivalent:
\begin{itemize}
\item[{\rm(i)}] $(\g,\g^*,\alpha,\beta)$ is a mock-pre-Lie bialgebra.

\item[{\rm(ii)}]$(\g,\g^*;(-L^*-R^*, R^*),(-\huaL^*-\huaR^*, \huaR^*))$ is
a matched pair of mock-pre-Lie algebras.

\item[{\rm(iii)}] $(\huaG=\g\oplus\g^*,\g,\g^*)$ is a Manin triple of mock-pre-Lie algebras, where the invariant skew-symmetric bilinear form on $\g\oplus\g^*$ is given by \eqref{omega}.
\item[{\rm(iv)}] $(\g_{c},\g_{c}^*;-L^*,-\huaL^*)$ is a matched pair of mock-Lie algebras.

 \item[\rm(v)]$(\g_{c}\oplus \g_{c}^*,\{\cdot,\cdot\},\omega)$ is a phase space of the sub-adjacent mock-Lie algebra $(\g_c,\{\cdot,\cdot\}_{\g}).$
\end{itemize}
\end{thm}

\begin{proof}
By Theorem \ref{equivalent-match pre-Lie}, (ii) is equivalent to (iv).
First we prove that (iv) is equivalent to (i). Let $(\g,\cdot_{\g})$ and $(\g^*,\cdot_{\g^*})$ be two mock-pre-Lie algebras.  Let $(\g_{c},\g_{c}^*;-L^*,-\huaL^*)$ be a matched pair of mock-Lie algebras. By \eqref{matched-ML1}, taking pair with $y\in \g,$ we have
\begin{eqnarray*}
&&\langle-L^*_{x}\{\xi,\eta\}_{\g^*}+(-L^*_{x}\xi)\cdot_{\g^*}\eta+\eta\cdot_{\g^*}(-L^*_{x}\xi)+\xi\cdot_{\g^*}(-L^*_{x}\eta)\\
&&+(-L^*_{x}\eta)\cdot_{\g^*}\xi+L^*_{(\huaL^*_{\xi}x)}\eta+L^*_{(\huaL^*_{\eta}x)}\xi,y\rangle\\
&=&\langle\{\xi,\eta\}_{\g^*},x\cdot_{\g}y\rangle-\langle\xi,x\cdot_{\g}(\huaR^*_{\eta}y)\rangle-\langle\xi,x\cdot_{\g}(\huaL^*_{\eta}y)\rangle
-\langle\eta,x\cdot_{\g}(\huaL^*_{\xi}y)\rangle\\
&&-\langle\eta,x\cdot_{\g}(\huaR^*_{\xi}y)\rangle
-\langle\eta,(\huaL^*_{\xi}x)\cdot_{\g}y\rangle-\langle\xi,(\huaL^*_{\eta}x)\cdot_{\g}y\rangle\\
&=&\langle\beta\{\xi,\eta\}_{\g^*}+({\Id}\otimes (\huaR_{\eta}+\huaL_{\eta}))\beta(\xi)+({\Id}\otimes (\huaL_{\xi}+\huaR_{\xi}))\beta(\eta)
\\
&&+(\huaL_{\eta}\otimes {\Id})\beta(\xi)+(\huaL_{\xi}\otimes {\Id})\beta(\eta),x\otimes y\rangle\\
&=&0,
\end{eqnarray*}
which implies that $\beta\{\xi,\eta\}_{\g^*}=-({\Id}\otimes (\huaR_{\eta}+\huaL_{\eta})+\huaL_{\eta}\otimes {\Id})\beta(\xi)-({\Id}\otimes (\huaL_{\xi}+\huaR_{\xi})+\huaL_{\xi}\otimes {\Id})\beta(\eta).$ Similarly, by \eqref{matched-ML2}, taking pair with $\eta\in \g^*,$ we also have
\begin{eqnarray*}
&&\langle-\huaL^*_{\xi}\{x,y\}_{\g}+(-\huaL^*_{\xi}x)\cdot_{\g}y+y\cdot_{\g}(-\huaL^*_{\xi}x)+x\cdot_{\g}(-\huaL^*_{\xi}y)\\
&&+(-\huaL^*_{\xi}y)\cdot_{\g}x+\huaL^*_{(L^*_{x}\xi)}y+\huaL^*_{(L^*_{y}\xi)}x,\eta\rangle\\
&=&\langle\{x,y\}_{\g},\xi\cdot_{\g^*}\eta\rangle-\langle x,\xi\cdot_{\g^*}(R^*_{y}\eta)\rangle-\langle x,\xi\cdot_{\g^*}(L^*_{y}\eta)\rangle
-\langle y,\xi\cdot_{\g^*}(L^*_{x}\eta)\rangle\\
&&-\langle y,\xi\cdot_{\g^*}(R^*_{x}\eta)\rangle-\langle y,(L^*_{x}\xi)\cdot_{\g^*}\eta\rangle-\langle x,(L^*_{y}\xi)\cdot_{\g^*}\eta\rangle\\
&=&\langle\alpha\{x,y\}_{\g}+({\Id}\otimes (R_{y}+L_{y}))\alpha(x)+({\Id}\otimes (L_{x}+R_{x}))\alpha(y)\\
&&+(L_{y}\otimes {\Id})\alpha(x)+(L_{x}\otimes {\Id})\alpha(y),\xi\otimes\eta\rangle\\
&=&0,
\end{eqnarray*}
which implies that $\alpha\{x,y\}_{\g}=-({\Id}\otimes (R_{y}+L_{y})+L_{y}\otimes {\Id})\alpha(x)-({\Id}\otimes (L_{x}+R_{x})+L_{x}\otimes {\Id})\alpha(y).$
Thus, $(\g,\g^*,\alpha,\beta)$ is a mock-pre-Lie bialgebra if and only if $(\g_{c},\g_{c}^*;-L^*,-\huaL^*)$ is
a matched pair of mock-Lie algebras.

Next we prove that (ii) is equivalent to (iii).
Let $(\g,\g^*;(-L^*-R^*, R^*),(-\huaL^*-\huaR^*, \huaR^*))$ be
a matched pair of mock-pre-Lie algebras. Then $(\g\oplus\g^*,\cdot_{\bowtie}) $ is a  mock-pre-Lie algebra, where$\cdot_{\bowtie}$ is given by \eqref{matched-pair-mock-pre}. We only need to prove that $\omega$ satisfies the invariant condition \eqref{quadratic-mock-pre}. For all $x, y, z \in \g$ and $\xi,\eta,\zeta\in \g^*,$ we have
\begin{eqnarray*}
  &&\omega((x+\xi)\cdot_{\bowtie}(y+\eta),z+\zeta)\\
  &=&\omega(x\cdot_{\g}y-(\huaL_{u}^*+\huaR_{\xi}^*)y+\huaR_{\eta}^*x+\xi\cdot_{\g^*}\eta-(L_{x}^*+R_{x}^*)\eta+R_{y}^*\xi,z+\zeta)\\
   &=&\langle \xi\cdot_{\g^*}\eta,z\rangle+\langle \eta,\{x,z\}_{\g}\rangle-\langle \xi,z\cdot_{\g}y\rangle-\langle \zeta,x\cdot_{\g}y\rangle-\langle \{\xi,\zeta\}_{\g^*},y\rangle+\langle \zeta\cdot_{\g^*}\eta,x\rangle\\
   &=&\langle \eta,\{x,z\}_{\g}-\huaL_{\xi}^*z-\huaL_{\zeta}^*x\rangle-\langle\{\xi,\zeta\}_{\g^*}-L_{z}^*\xi-L_{x}^*\zeta,y\rangle\\
    &=&\omega( y+\eta,\{x,z\}_{\g}-\huaL_{\xi}^*z-\huaL_{\zeta}^*x+\{\xi,\zeta\}_{\g^*}-L_{z}^*\xi-L_{x}^*\zeta)\\
    &=&\omega( y+\eta,(x+\xi)\cdot_{\bowtie}(z+\zeta)+(z+\zeta)\cdot_{\bowtie}(x+\xi)).
\end{eqnarray*}
Thus, $\omega$ satisfies the invariant condition \eqref{quadratic-mock-pre}.

Conversely, let $(\g\oplus\g^*, \g,\g^*)$ be a Manin triple of mock-pre-Lie algebras, where the invariant skew-symmetric bilinear form given by \eqref{omega}.
By Proposition \ref{pro-matched-pair-mock-pre}, we have $(\g, \g^*; (\rho_{\g}, \mu_{\g}), (\rho_{\g^*}, \mu_{\g^*}))$ is a matched pair of mock-pre-Lie algebras, where $(\rho_{\g}, \mu_{\g})$ and $(\rho_{\g^*}, \mu_{\g^*})$ are determined by \eqref{mock-pre-Lie-rho-mu}.
More precisely, for all $x,y\in \g,\xi,\eta\in \g^*$, by \eqref{quadratic-mock-pre}, we have
\begin{eqnarray*}
  \langle \rho_{\g}(x)\xi,y\rangle=\omega(x\cdot_{\huaG}\xi,y)
  =\omega(\xi,x\cdot_{\g}y+y\cdot_{\g}x)=\omega(\xi,L_{x}y+R_{x}y)=\langle -L^*_{x}\xi-R^*_{x}\xi,y\rangle,
\end{eqnarray*}
which implies that $\rho_{\g}=-L^*-R^*.$ By \eqref{quadratic-mock-pre-2}, we have
\begin{eqnarray*}
  \langle \mu_{\g}(x)\xi,y\rangle=\omega(\xi\cdot_{\huaG}x,y)
  =\omega(y\cdot_{\g}x,\xi)=-\langle\xi,R_{x}y\rangle
  =\langle R^*_{x}\xi,y\rangle,
\end{eqnarray*}
which implies that $\mu_{\g}=R^*.$ Similarly, we can obtain $\rho_{\g^*}=-\huaL^*-\huaR^*$ and $\mu_{\g^*}=\huaR^*,$
thus the matched pair of mock-pre-Lie algebras is $(\g,\g^*;(-L^*-R^*, R^*),(-\huaL^*-\huaR^*, \huaR^*)).$ By Theorem \ref{phase-space-Manin-triple}, we have  (iii) is equivalent to (v).
Thus we finish the proof.
\end{proof}

\begin{defi}
Let $(\g,\g^*,\alpha_{\g},\beta_{\g})$ and $(\h,\h^*,\alpha_{\h},\beta_{\h})$ be two mock-pre-Lie bialgebras. A linear map $\varphi:\g \to \h$ is called {\bf a homomorphism of mock-pre-Lie  bialgebras}, if $\varphi:\g\rightarrow \h$ is a homomorphism of mock-pre-Lie algebras such that $\varphi^*:\h^* \rightarrow\g^*$ is also a homomorphism of mock-pre-Lie  algebras, or equivalently, $\varphi$ satisfying
\begin{eqnarray}
(\varphi \otimes \varphi)\circ \alpha_{\g}=\alpha_{\h} \circ \varphi, \quad (\varphi^* \otimes \varphi^*)\circ \beta_{\h}=\beta_{\g} \circ \varphi^*.
\end{eqnarray}
Furthermore, if $\varphi:\g \rightarrow \h$ is a linear isomorphism of vector spaces, then $\varphi$ is called {\bf an isomorphism of mock-pre-Lie bialgebras.}
\end{defi}

\begin{ex}\label{mock-pre-Lie-example}
Let $(\g, \g^*)$ be a mock-pre-Lie bialgebra. Then $(\g^*, \g)$ is also a mock-pre-Lie bialgebra.
\end{ex}

\begin{ex}{\rm
Let $(\g,\cdot_{\g})$ be a mock-pre-Lie algebra and the mock-pre-Lie algebra structure on $\g^*$ be trivial, then in this case $(\g,\g^*,0,\beta)$ is a mock-pre-Lie bialgebra, which corresponds to the mock-pre-Lie algebra $\g \ltimes_{-L^*-R^*,R^*} \g^*.$
}
\end{ex}

\subsection{Coboundary mock-pre-Lie bialgebras and the mock-pre-Lie Yang-Baxter equation}
In this subsection, we consider a special class of mock-pre-Lie bialgebras called coboundary mock-pre-Lie bialgebras and introduce the notion of mock-pre-Lie classical Yang-Baxter equation.
\begin{defi}
A mock-pre-Lie bialgebra $(\g,\g^*,\alpha, \beta)$ is called \textbf{coboundary} if   there exists an element $r\in \g\otimes \g$ such that for any $x\in \g$,
\begin{equation}\label{coboundary}
\alpha(x)=({\Id}\otimes R_{x}+{\Id}\otimes L_{x}-L_{x}\otimes {\Id})r.
\end{equation}
\end{defi}
Let $\g$ be a mock-pre-Lie algebra whose product is given by $\beta:\g\otimes\g\rightarrow\g$ and $r\in \g\otimes\g.$ Suppose that
$\alpha$ satisfies the equation \eqref{coboundary}, then we can directly verify that \eqref{alpha-condition} is satisfied.
Therefore, $(\g,\g^*,\alpha, \beta)$ is a mock-pre-Lie bialgebra if and only if the following conditions hold:
\begin{itemize}
\item[{\rm(i)}] $\alpha^* : \g^*\otimes \g^* \rightarrow \g^*$ is a mock-pre-Lie algebra structure on $\g^*;$
\item[{\rm(ii)}] $\beta : \g^* \rightarrow \g^* \otimes \g^*$ satisfies the compatibility condition \eqref{beta-condition}.
\end{itemize}
\emptycomment{
\hou{Verify that \eqref{alpha-condition} is satisfied}
\begin{proof}
\begin{eqnarray*}
  LHS=\alpha\{x,y\}_{\g}&=&({\Id}\otimes R_{\{x,y\}_{\g}}+{\Id}\otimes L_{\{x,y\}_{\g}}-L_{\{x,y\}_{\g}}\otimes {\Id})(a\otimes b)\\
  &=&a\otimes\{b,\{x,y\}_{\g}\}_{\g}-(x\cdot_{\g}y)\cdot_{\g}a\otimes b-(y\cdot_{\g}x)\cdot_{\g}a\otimes b
\end{eqnarray*}
\begin{eqnarray*}
RHS&=&-({\Id}\otimes (R_{y}+L_{y})+L_{y}\otimes {\Id})\alpha(x)-({\Id}\otimes (L_{x}+R_{x})+L_{x}\otimes {\Id})\alpha(y)\\
&=&-a\otimes\{x,\{y,b\}_{\g}\}_{\g}-a\otimes\{y,\{b,x\}_{\g}\}_{\g}+[y\cdot_{\g}(x\cdot_{\g}a)+x\cdot_{\g}(y\cdot_{\g}a)]\otimes b
\end{eqnarray*}
\end{proof}}

\begin{pro}
Let $(\g, \cdot_{\g})$ be a mock-pre-Lie algebra whose product is given by $\beta^*: \g \otimes \g \rightarrow \g$ and $r \in \g \otimes \g$. Suppose there exists a mock-pre-Lie algebra structure $\cdot_{\g^*}$ on $\g^*$
given by $\alpha^* : \g^* \otimes \g^* \rightarrow \g^*$, where $\alpha$ is given by \eqref{coboundary}. Then $\beta : \g^* \rightarrow \g^* \otimes \g^*$ satisfies the compatibility condition \eqref{beta-condition} if and only if $r$ satisfies
\begin{equation}
[P_{-}(x \cdot_{\g} y) + P_{+}(x)P_{-}(y)](r- \tau(r))=0, \quad  \forall x, y\in \g,
\end{equation}
where $P_{\pm}(x) = L_x \otimes {\Id} \pm {\Id} \otimes L_x,$ the linear map $\tau:\g\otimes\g\rightarrow\g\otimes\g$ is the flip operator given by $\tau(x\otimes y)=y\otimes x.$
\end{pro}
\begin{proof}
  Assume that $\beta:\g^*\rightarrow\g^*\otimes\g^*$ satisfies the compatibility condition, for all $x,y\in \g, \xi,\eta\in \g^*,$ we have
  \begin{eqnarray*}
&&\langle \beta\{\xi,\eta\}_{\g^*}+({\Id}\otimes (\huaR_{\eta}+\huaL_{\eta})+\huaL_{\eta}\otimes {\Id})\beta(\xi)+({\Id}\otimes (\huaL_{\xi}+\huaR_{\xi})+\huaL_{\xi}\otimes {\Id})\beta(\eta),x\otimes y\rangle\\
&=&\langle \{\xi,\eta\}_{\g^*},x\cdot_{\g}y\rangle-\langle \xi,x\cdot_{\g} \huaR^*_{\eta}y\rangle-\langle \xi,x\cdot_{\g} \huaL^*_{\eta}y\rangle-\langle \xi, \huaL^*_{\eta}x\cdot_{\g}y\rangle\\
&&-\langle \eta, x\cdot_{\g}\huaL^*_{\xi}y\rangle-\langle \eta, x\cdot_{\g}\huaR^*_{\xi}y\rangle-\langle \eta,\huaL^*_{\xi}x \cdot_{\g}y\rangle\\
&=&\langle \xi\otimes \eta,\alpha(x\cdot_{\g}y)+\tau(\alpha(x\cdot_{\g}y))+(L_{x}\otimes{\Id})\alpha(y)+(L_x\otimes {\Id})\tau(\alpha(y))\\
&&+(R_{y}\otimes {\Id})\tau(\alpha(x))+({\Id}\otimes L_{x})\alpha(y)+({\Id}\otimes L_{x})\tau(\alpha(y))+({\Id}\otimes R_{y})\alpha(x)\rangle\\
&=&\langle \xi\otimes \eta, -(L_{x\cdot_{\g}y}\otimes {\Id}-{\Id}\otimes R_{x\cdot_{\g}y}-{\Id}\otimes L_{x\cdot_{\g}y}+L_xL_y\otimes {\Id}-L_x\otimes R_y-L_x\otimes L_y\\
&&+L_y\otimes L_x-{\Id}\otimes L_xR_y-{\Id}\otimes L_xL_y+L_x\otimes R_y-{\Id}\otimes R_yR_x-{\Id}\otimes R_yL_x)r\rangle
\\
&&+\langle \xi\otimes \eta,-({\Id}\otimes L_{x\cdot_{\g}y}-R_{x\cdot_{\g}y}\otimes {\Id}-L_{x\cdot_{\g}y}\otimes {\Id}
+L_x\otimes L_y-L_xR_y\otimes{\Id} -L_xL_y\otimes{\Id}\\
&&+R_y\otimes L_x-R_yR_x\otimes{\Id}-R_yL_x\otimes{\Id}+{\Id}\otimes L_xL_y-R_y\otimes L_x-L_y\otimes L_x)\tau(r)\rangle\\
&=&-\langle \xi\otimes \eta, (L_{x\cdot_{\g}y}\otimes {\Id}-{\Id}\otimes L_{x\cdot_{\g}y}+L_{x}L_{y}\otimes {\Id}-{\Id}\otimes L_{x}L_{y}-L_{x}\otimes L_{y}+L_{y}\otimes L_{x})(r-\tau(r))\rangle\\
&=&-\langle \xi\otimes \eta,[P_{-}(x\cdot_{\g} y)+P_{+}(x)P_{-}(y)](r-\tau(r))\rangle\\
&=&0.
  \end{eqnarray*}
  Thus, the proof is completed.
\end{proof}

For any linear map $\alpha:\g\rightarrow \g\otimes \g$, let $J_\alpha:\g\rightarrow \g\otimes \g\otimes \g$ be a linear map given by
\begin{equation}\label{mock-pre-alpha}
J_{\alpha}(x)=(\alpha\otimes {\Id})\alpha(x)+({\Id}\otimes \alpha)\alpha(x)+(\tau\otimes {\Id})(\alpha\otimes {\Id})\alpha(x)+(\tau\otimes {\Id})({\Id}\otimes \alpha)\alpha(x),\quad \forall x\in \g.
\end{equation}

\begin{lem}
Let $\g$ be a vector space and $\alpha:\g
\rightarrow \g\otimes\g$ be a linear map. Then $\alpha^*:\g^*\otimes
\g^*\rightarrow \g^*$ defines a mock-pre-Lie algebra structure on
$\g^*$ if and only if $J_\alpha=0$.
\end{lem}
\begin{proof}
  For all $x\in \g, \xi,\eta,\zeta\in \g^*,$ we have
  \begin{eqnarray*}
   &&\langle(\xi\cdot_{\g^*}\eta)\cdot_{\g^*}\zeta+\xi\cdot_{\g^*}(\eta\cdot_{\g^*}\zeta)
   +(\eta\cdot_{\g^*}\xi)\cdot_{\g^*}\zeta+\eta\cdot_{\g^*}(\xi\cdot_{\g^*}\zeta),x\rangle\\
   &=&\langle \xi\otimes \eta\otimes \zeta,(\alpha\otimes {\Id})\alpha(x)+({\Id}\otimes \alpha)\alpha(x)+(\tau\otimes {\Id})(\alpha\otimes {\Id})\alpha(x)+(\tau\otimes {\Id})({\Id}\otimes \alpha)\alpha(x)\rangle\\
  &=&\langle \xi\otimes \eta\otimes \zeta,J_{\alpha}(x)\rangle,
  \end{eqnarray*}
which implies that $(\g^*,\alpha^*=\cdot_{\g^*})$ is a mock-pre-Lie algebra if and only if $J_\alpha=0$.
\end{proof}

Let $(\g,\cdot)$ be a mock-pre-Lie algebra and $r=\sum_{i}a_i\otimes b_i\in \g\otimes \g$. Set
\begin{equation}
r_{12}=\sum_ia_i\otimes b_i\otimes 1;r_{21}=\sum_i b_i\otimes a_i\otimes 1;\;
r_{13}=\sum_{i}a_i\otimes 1\otimes b_i;\;r_{23}=\sum_i1\otimes a_i\otimes b_i,
\end{equation}
where $1$ is a unit element if $(\g,\cdot)$  is unital, or a symbol playing a similar role in the non-unital cases. The operation between two $r_{ij}$ are defined by
{\small\begin{align}
    &r_{12}\cdot r_{13}=\sum_{ij}a_i\cdot a_j\otimes b_i\otimes b_j,\ \ \{r_{13},r_{23}\}=\sum_{ij}a_i\otimes a_j\otimes \{b_i,b_j\}.
\end{align}}
All other cases are defined in the same way.
\begin{pro}\label{mock-r-matrix}
Let $(\g,\cdot_{\g})$ be a mock-pre-Lie algebra. Let $r=\sum_{i}a_i\otimes b_i\in \g\otimes \g$, where
$a_i,b_i\in \g$. If $\alpha:\g\rightarrow \g\otimes \g$ is defined by \eqref{coboundary}, then we have
\begin{equation}\label{eqn4}
J_\alpha(x)=Q(x)[[r,r]]+\sum_{j}[P_{-}(x\cdot_{\g} a_j)+P_{+}(x)P_{-}(a_j)](r-\tau(r))\otimes b_j,
\end{equation}
where
\begin{equation}\label{eqno 5.5}
[[r,r]]=-r_{13}\cdot r_{12}-r_{23}\cdot r_{21}+\{r_{23},r_{12}\}+\{r_{13},r_{21}\}-\{r_{13},r_{23}\}
\end{equation}
and $Q(x)=L_x\otimes {\Id}\otimes {\Id}+{\Id}\otimes L_x\otimes {\Id}+{\Id}\otimes {\Id}\otimes R_x+{\Id}\otimes{\Id}\otimes L_x$,
$P_{\pm}(x)=L_x\otimes {\Id}\pm{\Id}\otimes L_x$ for any $x\in \g$.
\end{pro}
\begin{proof}
By \eqref{coboundary} and \eqref{mock-pre-alpha}, for all $x\in \g,$ we have
\begin{eqnarray*}
J_\alpha(x)&=&(\alpha\otimes {\Id})\alpha(x)+({\Id}\otimes \alpha)\alpha(x)+(\tau\otimes {\Id})(\alpha\otimes {\Id})\alpha(x)+(\tau\otimes {\Id})({\Id}\otimes \alpha)\alpha(x)\\
&=&\sum_{i,j}(x\cdot_{\g}a_i)\cdot_{\g}a_j\otimes b_j\otimes b_i\uwave{-a_j\otimes b_j\cdot_{\g}(x\cdot_{\g}a_i)\otimes b_i}-a_j\otimes(x\cdot_{\g}a_i)\cdot_{\g}b_j\otimes b_i\\
&&-a_i\cdot_{\g}a_j\otimes b_j\otimes b_i\cdot_{\g}x+a_j\otimes b_j\cdot_{\g}a_i\otimes b_i\cdot_{\g}x+a_j\otimes a_i\cdot_{\g}b_j\otimes b_i\cdot_{\g}x\\
&&-a_i\cdot_{\g}a_j\otimes b_j\otimes x\cdot_{\g}b_i+a_j\otimes b_j\cdot_{\g}a_i\otimes x\cdot_{\g}b_i+a_j\otimes a_i\cdot_{\g}b_j\otimes x\cdot_{\g}b_i\\
&&+(x\cdot_{\g}a_i)\otimes b_j\cdot_{\g}a_j\otimes b_j-x\cdot_{\g}a_i\otimes a_j\otimes b_j\cdot_{\g}b_i-x\cdot_{\g}a_i\otimes a_j\otimes b_i\cdot_{\g}b_j\\
&&\uwave{-a_i\otimes (b_i\cdot_{\g}x)\cdot_{\g}a_j\otimes b_j}+\dashuline{a_i\otimes a_j\otimes b_j\cdot_{\g}(b_i\cdot_{\g}x)}+\uline{a_i\otimes a_j\otimes
(b_i\cdot_{\g}x)\cdot_{\g} b_j}\\
&&\uwave{-a_i\otimes (x\cdot_{\g}b_i)\cdot_{\g}a_j\otimes b_j}+\uuline{a_i\otimes a_j\otimes b_j\cdot_{\g}(x\cdot_{\g}b_i)}+\uline{a_i\otimes a_j\otimes (x\cdot_{\g}b_i)\cdot_{\g}b_j}\\
&&+b_j\otimes (x\cdot_{\g}a_i)\cdot_{\g}a_j\otimes b_i\dotuline{-b_j\cdot_{\g}(x\cdot_{\g}a_i)\otimes a_j\otimes b_i}-(x\cdot_{\g}a_i)\cdot_{\g}b_j\otimes a_j\otimes b_i\\
&&-b_j\otimes a_i\cdot_{\g}a_j\otimes b_i\cdot_{\g}x+b_j\cdot_{\g}a_i\otimes a_j\otimes b_i\cdot_{\g}x+a_i\cdot_{\g}b_j\otimes a_j\otimes b_i\cdot_{\g}x\\
&&-b_j\otimes a_i\cdot_{\g}a_j\otimes x\cdot_{\g}b_i+b_j\cdot_{\g}a_i\otimes a_j\otimes x\cdot_{\g}b_i+a_i\cdot_{\g}b_j\otimes a_j\otimes x\cdot_{\g}b_i\\
&&+b_j\cdot_{\g}a_j\otimes x\cdot_{\g}a_i\otimes b_j-a_j\otimes x\cdot_{\g}a_i\otimes b_j\cdot_{\g}b_i-a_j\otimes x\cdot_{\g}a_i\otimes b_i\cdot_{\g}b_j\\
&&\dotuline{-(b_i\cdot_{\g}x)\cdot_{\g}a_j\otimes a_i\otimes b_j} +\dashuline{a_j\otimes a_i\otimes b_j\cdot_{\g}(b_i\cdot_{\g}x)}+\uuline{a_j\otimes a_i\otimes
(b_i\cdot_{\g}x)\cdot_{\g} b_j}\\
&&\dotuline{-(x\cdot_{\g}b_i)\cdot_{\g}a_j\otimes a_i \otimes b_j}+\uline{a_j\otimes a_i\otimes b_j\cdot_{\g}(x\cdot_{\g}b_i)}+\uuline{a_j\otimes a_i\otimes (x\cdot_{\g}b_i)\cdot_{\g}b_j}\\
&=&\sum_{i,j}\Big(-a_i\cdot_{\g}a_j\otimes b_j\otimes b_i\cdot_{\g}x+a_j\otimes b_j\cdot_{\g}a_i\otimes b_i\cdot_{\g}x+a_j\otimes a_i\cdot_{\g}b_j\otimes b_i\cdot_{\g}x\\
&&-a_i\cdot_{\g}a_j\otimes b_j\otimes x\cdot_{\g}b_i+a_j\otimes b_j\cdot_{\g}a_i\otimes x\cdot_{\g}b_i+a_j\otimes a_i\cdot_{\g}b_j\otimes x\cdot_{\g}b_i\\
&&+(x\cdot_{\g}a_i)\otimes b_j\cdot_{\g}a_j\otimes b_j-x\cdot_{\g}a_i\otimes a_j\otimes b_j\cdot_{\g}b_i-x\cdot_{\g}a_i\otimes a_j\otimes b_i\cdot_{\g}b_j\\
&&-b_j\otimes a_i\cdot_{\g}a_j\otimes b_i\cdot_{\g}x+b_j\cdot_{\g}a_i\otimes a_j\otimes b_i\cdot_{\g}x+a_i\cdot_{\g}b_j\otimes a_j\otimes b_i\cdot_{\g}x\\
&&-b_j\otimes a_i\cdot_{\g}a_j\otimes x\cdot_{\g}b_i+b_j\cdot_{\g}a_i\otimes a_j\otimes x\cdot_{\g}b_i+a_i\cdot_{\g}b_j\otimes a_j\otimes x\cdot_{\g}b_i\\
&&+b_j\cdot_{\g}a_j\otimes x\cdot_{\g}a_i\otimes b_j-a_j\otimes x\cdot_{\g}a_i\otimes b_j\cdot_{\g}b_i-a_j\otimes x\cdot_{\g}a_i\otimes b_i\cdot_{\g}b_j\\
&&-x\cdot_{\g}(a_j\cdot_{\g}a_i)\otimes b_i\otimes b_j+x\cdot_{\g}(a_j\cdot_{\g}b_i)\otimes a_i\otimes b_j+(x\cdot_{\g}a_i)\otimes(a_j\cdot_{\g}b_i)\otimes b_j\\
&&-(x\cdot_{\g}b_i)\otimes(a_j\cdot_{\g}a_i)\otimes b_j-(a_j\cdot_{\g}a_i)\otimes(x\cdot_{\g}b_i)\otimes b_j
+(a_j\cdot_{\g}b_i)\otimes(x\cdot_{\g}a_i)\otimes b_j\\
&&+a_i\otimes x\cdot_{\g}(a_j\cdot_{\g}b_i)\otimes b_j-b_i\otimes x\cdot_{\g}(a_j\cdot_{\g}a_i)\otimes b_j+\uwave{a_j\otimes x\cdot_{\g}(b_j\cdot_{\g}a_i)\otimes b_i}\\
&&+\dotuline{x\cdot_{\g}(b_i\cdot_{\g} a_j)\otimes a_i\otimes b_j}\uline{-a_i\otimes a_j\otimes x\cdot_{\g}(b_i\cdot_{\g} b_j)}\uuline{-a_i\otimes a_j\otimes x\cdot_{\g}(b_j\cdot_{\g} b_i)}\\
&&\dashuline{-a_i\otimes a_j\otimes (b_j\cdot_{\g}b_i)\cdot_{\g}x-a_i\otimes a_j\otimes (b_i\cdot_{\g} b_j)\cdot_{\g}x}\Big)\\
&&+\sum_{i,j}\Big((x\cdot_{\g}a_j)\cdot_{\g}a_i\otimes b_i\otimes b_j-a_i\otimes(x\cdot_{\g}a_j)\cdot_{\g}b_i\otimes b_j+b_i\otimes (x\cdot_{\g}a_j)\cdot_{\g}a_i\otimes b_j\\
&&-(x\cdot_{\g}a_j)\cdot_{\g}b_i\otimes a_i\otimes b_j+x\cdot_{\g}(a_j\cdot_{\g}a_i)\otimes b_i\otimes b_j-x\cdot_{\g}(a_j\cdot_{\g}b_i)\otimes a_i\otimes b_j\\
&&-(x\cdot_{\g}a_i)\otimes(a_j\cdot_{\g}b_i)\otimes b_j+(x\cdot_{\g}b_i)\otimes(a_j\cdot_{\g}a_i)\otimes b_j+(a_j\cdot_{\g}a_i)\otimes(x\cdot_{\g}b_i)\otimes b_j\\
&&-(a_j\cdot_{\g}b_i)\otimes(x\cdot_{\g}a_i)\otimes b_j-a_i\otimes x\cdot_{\g}(a_j\cdot_{\g}b_i)\otimes b_j+
b_i\otimes x\cdot_{\g}(a_j\cdot_{\g}a_i)\otimes b_j\Big)\\
&=&\Big(L_x\otimes {\Id}\otimes {\Id}+{\Id}\otimes L_x\otimes {\Id}+{\Id}\otimes {\Id}\otimes R_x+{\Id}\otimes{\Id}\otimes L_x\Big)\Big(-r_{13}\cdot r_{12}-r_{23}\cdot r_{21}\\
&&+\{r_{23},r_{12}\}+\{r_{13},r_{21}\}-\{r_{13},r_{23}\}\Big)+\sum_{j}\Big(L_{x\cdot_{\g}a_j}\otimes {\Id}-{\Id}\otimes L_{x\cdot_{\g}a_j}+L_{x}L_{a_j}\otimes {\Id}\\
&&-{\Id}\otimes L_{x}L_{a_j}-L_{x}\otimes L_{a_j}+L_{a_j}\otimes L_{x}\Big)(r-\tau(r))\otimes b_j\\
&=&Q(x)[[r,r]]+\sum_{j}[P_{-}(x\cdot_{\g} a_j)+P_{+}(x)P_{-}(a_j)](r-\tau(r))\otimes b_j.
\end{eqnarray*}
Therefore, we finished the proof.
\end{proof}

With the discussion above, we have the following result.

\begin{thm}\label{mock-preLiebia-condition}
Let $\g$ be a mock-pre-Lie algebra and $r\in\g\otimes \g$. Then the map $\alpha$ defined by equation \eqref{coboundary} induces
a mock-pre-Lie algebra structure on $\g^*$ such that $(\g,\g^*)$ is a
mock-pre-Lie bialgebra if and only if the following two conditions
are satisfied:
\begin{itemize}
\item[{\rm(i)}] $[P_{-}(x\cdot_{\g} y)+P_{+}(x)P_{-}(y)](r-\tau(r))=0$ \quad $\forall x,y\in \g$;
\item[{\rm(ii)}]$Q(x)[[r,r]]=0$,
\end{itemize}
\noindent where for all $x\in \g,$ $P_{\pm}(x),Q(x)$ and $[[r,r]]$ are defined by Proposition \ref{mock-r-matrix}.
\end{thm}

The equation $[[r,r]]=-r_{13}\cdot r_{12}-r_{23}\cdot r_{21}+\{r_{23},r_{12}\}+\{r_{13},r_{21}\}-\{r_{13},r_{23}\}=0$ is called the {\bf mock-pre-Lie classical Yang-Baxter equation}, as an analogue of the classical Yang-Baxter equation.

\section{Quasi-triangular mock-pre-Lie bialgebras and relative Rota-Baxter operators}\label{sec:III}
In this section, we introduce the notions of quasi-triangular mock-pre-Lie algebras and relative Rota-Baxter operators of arbitrary weight on mock-pre-Lie algebras. We show that a quasi-triangular  mock-pre-Lie algebra can give rise to a relative Rota-Baxter operator of weight $-1.$

Let $r\in \g\otimes\g,$ we define $r_{+},r_{-}:\g^*\rightarrow\g$ by
\begin{eqnarray}
\langle r_{+}(\xi),\eta\rangle=r(\xi,\eta),\quad \langle r_{-}(\xi),\eta\rangle=r(\eta,\xi),\quad \forall \xi,\eta \in \g^*.
\end{eqnarray}
 Explicitly, writing $r=\sum_{i}a_i\otimes b_i\in \g\otimes \g$, the corresponding map $r_{+}$ is
\begin{eqnarray}\label{basis-ccondition}
r_{+}(\xi)=\sum_{i}\langle\xi,a_i\rangle b_i,\quad \forall \xi \in \g^*.
\end{eqnarray}
Then the mock-pre-Lie algebra structure on $\g^*$ defined by \eqref{coboundary} in Theorem \ref{mock-preLiebia-condition} is given by
\begin{eqnarray}\label{dual-space-mockLie}
 \xi\cdot_{r} \eta=-(L^*_{r_{+}(\xi)}+R^*_{r_{+}(\xi)})\eta+R^*_{r_{-}(\eta)}\xi,  \quad \forall \xi,\eta \in \g^*.
\end{eqnarray}
 We will denote it by $(\g^*,\cdot_{r})$, or simply by $\g^*_{r}.$

Now we introduce the notion of $(L,L+R)$-invariance of $r\in \g\otimes\g$ , which is the main ingredient in the definition
of a quasi-triangular mock-pre-Lie bialgebra.
\begin{defi}
  Let $(\g, \cdot_{\g})$ be a mock-pre-Lie algebra. An element $r\in \g\otimes\g$ is called {\bf $(L,L+R)$-invariant} if
\begin{eqnarray}
\Big(L_{x}\otimes {\Id}-{\Id}\otimes (L_{x}+R_{x})\Big)r=0,\quad \forall x\in \g.
\end{eqnarray}
\end{defi}

\begin{lem}\label{mock-pre-invariant}
  Let $(\g, \cdot_{\g})$ be a mock-pre-Lie algebra and $r\in \g\otimes\g.$ Then $r$ is $(L,L+R)$-invariant if and only if
  \begin{eqnarray}
    r_{+}(L^*_x\xi)+\{x,r_{+}(\xi)\}_{c}=0.\quad \forall x\in \g,\xi\in\g^*.
  \end{eqnarray}
\begin{proof}
  For all $x\in \g,\xi,\eta \in \g^*,$ we have
  \begin{eqnarray}
      \langle r_{+}(L^*_x\xi)+\{x,r_{+}(\xi)\}_{c},\eta\rangle&=&\langle r,L^*_{x}\xi\otimes \eta\rangle+\langle L_x(r_{+}(\xi))+R_x(r_{+}(\xi)),\eta\rangle\\
   \nonumber&=&\langle r,L^*_{x}\xi\otimes \eta\rangle-\langle r,\xi\otimes(L^*_{x}+R^*_{x})\eta\rangle\\
   \nonumber&=&-\langle\Big(L_x\otimes{\Id}-{\Id}\otimes(L_x+R_x)\Big)r,\xi\otimes\eta\rangle,
  \end{eqnarray}
  which implies that $r$ is $(L,L+R)$-invariant if and only if $r_{+}(L^*_x\xi)+\{x,r_{+}(\xi)\}_{c}=0.$
\end{proof}
 \end{lem}

Next, for convenience in the subsequent discussion, we write $r=r_a+r_s,$ where $r_a$ denotes the skew-symmetric part of $r$
 and $r_s$ denotes the symmetric part of $r.$

\begin{pro}\label{deduce-condition}
  Let $(\g, \cdot_{\g})$ be a mock-pre-Lie algebra and $r\in \g\otimes\g.$ If the skew-symmetric part $r_{a}$ of $r$ is $(L,L+R)$-invariant, then we have $[P_{-}(x\cdot_{\g} y)+P_{+}(x)P_{-}(y)](r_{a})=0$, for all $x,y\in \g.$
\end{pro}
\begin{proof}
  Since the skew-symmetric part $r_{a}$ of $r$ is $(L,L+R)$-invariant, by Lemma \ref{mock-pre-invariant} and \eqref{defi-mock-preLie-equation}, we have
  \begin{eqnarray*}
    &&[P_{-}(x\cdot_{\g} y)+P_{+}(x)P_{-}(y)](r_a)(\xi,\eta)\\
      &=&\Big(L_{x\cdot_{\g}y}\otimes {\Id}-{\Id}\otimes L_{x\cdot_{\g}y}+L_{x}L_{y}\otimes {\Id}-{\Id}\otimes L_{x}L_{y}-L_{x}\otimes L_{y}+L_{y}\otimes L_{x}\Big)(r_a)(\xi,\eta)\\
  &=&-r_a(L^*_{x\cdot_{\g}y}\xi,\eta)+r_a(\xi,L^*_{x\cdot_{\g}y}\eta)+r_a(L^*_{y}(L^*_{x}\xi),\eta)
  -r_a(\xi,L^*_{y}(L^*_{x}\eta))-r_a(L^*_{x}\xi,L^*_{y}\eta)+r_a(L^*_{y}\xi,L^*_{x}\eta)\\
 &=&-\langle (r_a)_{+}(L^*_{x\cdot_{\g}y}\xi),\eta\rangle+\langle (r_a)_{+}(\xi),L^*_{x\cdot_{\g}y}\eta \rangle+\langle (r_a)_{+}(L^*_{y}(L^*_{x}\xi),\eta\rangle\\
 &&-\langle (r_a)_{+}(\xi),L^*_{y}(L^*_{x}\eta)\rangle-\langle (r_a)_{+}(L^*_{x}\xi),L^*_{y}\eta \rangle
 +\langle (r_a)_{+}(L^*_{y}\xi),L^*_{x}\eta\rangle\\
  &=&\langle\{x\cdot_{\g}y,(r_a)_{+}(\xi)\}_{c}-(x\cdot_{\g}y)\cdot_{\g}(r_a)_{+}(\xi)
  +\{y,\{x,(r_a)_{+}(\xi)\}_{c}\}_{c}\\&&-x\cdot_{\g}(y\cdot_{\g}(r_a)_{+}(\xi))
  -y\cdot_{\g}\{x,(r_a)_{+}(\xi)\}_{c}+x\cdot_{\g}\{y,(r_a)_{+}(\xi)\}_{c},\eta\rangle\\
  &=&0,
  \end{eqnarray*}
 which implies that $[P_{-}(x\cdot_{\g} y)+P_{+}(x)P_{-}(y)](r_{a})=0.$
\end{proof}

\begin{defi}
  Let $(\g, \cdot)$ be a mock-pre-Lie algebra. If $r\in \g\otimes\g$ satisfies $[[r,r]]=0$ and the skew-symmetric part $r_a$ is $(L,L+R)$-invariant, then the mock-pre-Lie bialgebra $(\g,\g_{r}^*)$ induced by $r$ is called a {\bf quasi-triangular mock-pre-Lie bialgebra}. Moreover, if $r$ is symmetric, it is called a {\bf triangular mock-pre-Lie bialgebra}.
\end{defi}

Denote by $I$ the operator
\begin{eqnarray}\label{sign-I}
I=r_{+}-r_{-}:\g^{\ast}\longrightarrow \g.
\end{eqnarray}

Note that $I^{\ast}=-I$. Actually, $\frac{1}{2} I$ is the contraction of the skew-symmetric part of $r$.
If $r$ is symmetric, then $I=0$.
\begin{pro}\label{invariance2}
Let $(\g, \cdot_{\g})$ be a mock-pre-Lie algebra and $r\in \g\otimes
\g.$ The skew-symmetric part $r_a$ of $r$ is $(L,L+R)$-invariant if
and only if $I\circ L_x^{\ast}=-(L_x+R_x) \circ I$, or $ L_x\circ I =-I
\circ (L_x^*+R_x^*),$  for all $x\in \g$,  where $I$ is
given by \eqref{sign-I}.
\end{pro}

\begin{proof}
By Lemma \ref{mock-pre-invariant}, the skew-symmetric part $r_a$ of $r$ is $(L,L+R)$-invariant if
and only if \vspace{1mm} $(r_a)_{+}(L^*_x\xi)+\{x,(r_a)_{+}(\xi)\}_{c}=0.$ Using $\frac{1}{2} I=(r_a)_{+},$ we can deduce that \vspace{1mm}the skew-symmetric part $r_a$ is $(L,L+R)$-invariant if
and only if $I\circ L_x^{\ast}=-(L_x+R_x) \circ I.$ Moreover, by $I^{\ast}=-I,$ we have $ L_x\circ I =-I\circ (L_x^*+R_x^*).$
\end{proof}

\begin{lem}\label{skew-invariant-lem}
Let $(\g, \cdot_{\g})$ be a mock-pre-Lie algebra and $r\in \g\otimes
\g.$ If the skew-symmetric part $r_a$ of $r$ is $(L,L+R)$-invariant, then
\begin{eqnarray}\label{skew-invariant}
 \langle\xi,x\cdot_{\g} (r_a)_{+}(\eta)\rangle+\langle\eta,\{x,(r_a)_{+}(\xi)\}_{c}\rangle=0 ,\quad \forall x\in \g, \xi,\eta \in \g^*.
\end{eqnarray}
Furthermore, the mock-pre-Lie algebra multiplication $\cdot_{r}$ on $\g^*$  given by \eqref{dual-space-mockLie} reduces to
\begin{eqnarray}\label{mockLie-induced-by-sym}
 \xi\cdot_{r} \eta=-(L^*_{(r_s)_{+}(\xi)}+R^*_{(r_s)_{+}(\xi)})\eta+R^*_{(r_s)_{+}(\eta)}\xi,  \quad \forall \xi,\eta \in \g^*.
\end{eqnarray}
\end{lem}

\begin{proof}
Since the skew-symmetric part $r_a$ of $r$ is $(L,L+R)$-invariant, by Proposition \ref{invariance2}, we have
  \begin{eqnarray*}
   \langle \xi,x\cdot_{\g} (r_a)_{+}(\eta)\rangle+\langle\eta,\{x,(r_a)_{+}(\xi)\}_{c}\rangle
    &=&-\langle L^*_{x}\xi,(r_a)_{+}(\eta) \rangle+\langle\eta,\{x,(r_a)_{+}(\xi)\}_{c}\rangle\\
     &=&\langle \eta, (r_a)_{+}(L^*_{x}\xi)+\{x,(r_a)_{+}(\xi)\}_{c}\rangle =0,
  \end{eqnarray*}
  which implies that \eqref{skew-invariant} holds.
  Moreover, by \eqref{skew-invariant}, we have
  \begin{eqnarray*}
\langle(L^*_{(r_a)_{+}(\xi)}+R^*_{(r_a)_{+}(\xi)})\eta+R^*_{(r_a)_{+}(\eta)}\xi,x\rangle=-( \langle\xi,x\cdot_{\g} (r_a)_{+}(\eta)\rangle+\langle\eta,\{x,(r_a)_{+}(\xi)\}_{c}\rangle)=0.
  \end{eqnarray*}

By \eqref{dual-space-mockLie}, we have
 \begin{eqnarray*}
   \xi\cdot_{r}\eta&=&-(L^*_{r_{+}(\xi)}+R^*_{r_{+}(\xi)})\eta+R^*_{r_{-}(\eta)}\xi\\
   &=&-(L^*_{(r_s)_{+}(\xi)}+R^*_{(r_s)_{+}(\xi)})\eta-(L^*_{(r_a)_{+}(\xi)}+R^*_{(r_a)_{+}(\xi)})\eta+R^*_{(r_s)_{+}(\eta)}\xi-R^*_{(r_a)_{+}(\eta)}\xi\\
   &=&-(L^*_{(r_s)_{+}(\xi)}+R^*_{(r_s)_{+}(\xi)})\eta+R^*_{(r_s)_{+}(\eta)}\xi.
 \end{eqnarray*}
The proof is complete.
\end{proof}

\begin{thm}\label{heart}
Let $(\g, \cdot_{\g})$ be a mock-pre-Lie algebra and $r\in \g\otimes \g$. Assume that the skew-symmetric part $r_{a}$ of $r$ is $(L,L+R)$-invariant. Then $r$ satisfies $[[r,r]]=0$ if and only if $(\g^{\ast},\cdot_{r})$ is a mock-pre-Lie algebra and
the linear maps $r_{+},r_{-}:\g^{*}\rightarrow \g$ are both mock-pre-Lie algebra homomorphisms from $(\g^{*},\cdot_{r})$ to $ (\g,\cdot_{\g})$.
\end{thm}

\begin{proof}
If $r$ satisfies $[[r,r]]=0,$ then by Theorem \ref{mock-preLiebia-condition} and Proposition \ref{deduce-condition}, $(\g^{\ast},\cdot_{r})$ is a mock-pre-Lie algebra.
By a direct calculation, for all $\xi,\eta,\zeta\in \g^*,$ we have
\begin{eqnarray*}
[[r_{s},r_{s}]](\xi,\eta,\zeta)
=\langle\eta,(r_{s})_{+}(\xi)\cdot_{\g}(r_{s})_{+}(\zeta)\rangle+\langle\xi,(r_{s})_{+}(\eta)\cdot_{\g}(r_{s})_{+}(\zeta)\rangle-\langle\zeta,\{(r_{s})_{+}(\xi),(r_{s})_{+}(\eta)\}_{c}\rangle.
\end{eqnarray*}
Moreover, by \eqref{skew-invariant}, we have
\begin{eqnarray*}
[[r_{a},r_{a}]](\xi,\eta,\zeta)=-\langle\zeta,\{(r_{a})_{+}(\xi),(r_{a})_{+}(\eta)\}_{c}\rangle
\end{eqnarray*}
and
\begin{eqnarray*}
[[r_{s},r_{a}]](\xi,\eta,\zeta)+[[r_{a},r_{s}]](\xi,\eta,\zeta)=0,
\end{eqnarray*}
which implies that
\begin{eqnarray*}
[[r,r]](\xi,\eta,\zeta)=[[r_{s}+r_{a},r_{s}+r_{a}]](\xi,\eta,\zeta)=[[r_{s},r_{s}]](\xi,\eta,\zeta)+[[r_{a},r_{a}]](\xi,\eta,\zeta).
\end{eqnarray*}

Next we show that the following equation holds
\begin{eqnarray}\label{eq-homomorphism}
  r_{+}(\xi\cdot_{r}\eta)-r_{+}(\xi)\cdot_{\g} r_{+}(\eta)=-[[r,r]](\xi,\cdot,\eta).
\end{eqnarray}
One the one hand, for the left part, by \eqref{mockLie-induced-by-sym}, we have
\begin{eqnarray*}
\langle\zeta,r_{+}(\xi\cdot_{r}\eta)-r_{+}(\xi)\cdot_{\g} r_{+}(\eta)\rangle&=&\langle\zeta,((r_s)_{+}+(r_a)_{+})(-L^*_{(r_s)_{+}(\xi)}\eta-R^*_{(r_s)_{+}(\xi)}\eta)+R^*_{(r_s)_{+}(\eta)}\xi)\rangle\\
&&-\langle\zeta,((r_s)_{+}+(r_a)_{+})(\xi)\cdot_{\g}((r_s)_{+}+(r_a)_{+})(\eta)\rangle\\
&=&\langle\zeta,(r_s)_{+}(-L^*_{(r_s)_{+}(\xi)}\eta-R^*_{(r_s)_{+}(\xi)}\eta)\rangle+\langle\zeta,(r_s)_{+}(R^*_{(r_s)_{+}(\eta)}\xi)\rangle\\
&&+\langle\zeta,(r_a)_{+}(-L^*_{(r_s)_{+}(\xi)}\eta-R^*_{(r_s)_{+}(\xi)}\eta)\rangle+\langle\zeta,(r_a)_{+}(R^*_{(r_s)_{+}(\eta)}\xi)\rangle\\
&&-\langle\zeta,(r_s)_{+}(\xi)\cdot_{\g}(r_s)_{+}(\eta)\rangle-\langle\zeta,(r_s)_{+}(\xi)\cdot_{\g}(r_a)_{+}(\eta)\rangle\\
&&-\langle\zeta,(r_a)_{+}(\xi)\cdot_{\g}(r_s)_{+}(\eta)\rangle-\langle\zeta,(r_a)_{+}(\xi)\cdot_{\g}(r_a)_{+}(\eta)\rangle\\
&=&\langle\eta,\{(r_s)_{+}(\xi),(r_s)_{+}(\zeta)\}_{c}\rangle-\langle\xi,(r_s)_{+}(\zeta)\cdot_{\g}(r_s)_{+}(\eta)\rangle\\
&&-\langle\eta,\{(r_a)_{+}(\zeta),(r_s)_{+}(\xi)\}_{c}\rangle+\langle\xi,(r_a)_{+}(\zeta)\cdot_{\g}(r_s)_{+}(\eta)\rangle\\
&&-\langle\zeta,(r_s)_{+}(\xi)\cdot_{\g}(r_s)_{+}(\eta)\rangle-\langle\zeta,(r_s)_{+}(\xi)\cdot_{\g}(r_a)_{+}(\eta)\rangle\\
&&-\langle\zeta,(r_a)_{+}(\xi)\cdot_{\g}(r_s)_{+}(\eta)\rangle-\langle\zeta,(r_a)_{+}(\xi)\cdot_{\g}(r_a)_{+}(\eta)\rangle
\end{eqnarray*}

One the other hand, we have
\begin{eqnarray*}
 [[r,r]](\xi,\zeta,\eta)&=&[[r_{s},r_{s}]](\xi,\zeta,\eta)+[[r_{a},r_{a}]](\xi,\zeta,\eta)\\
 &=&\langle\zeta,(r_{s})_{+}(\xi)\cdot_{\g}(r_{s})_{+}(\eta)\rangle+\langle\xi,(r_{s})_{+}(\zeta)\cdot_{\g}(r_{s})_{+}(\eta)\rangle\\
 &&-\langle\eta,\{(r_{s})_{+}(\xi),(r_{s})_{+}(\zeta)\}_{c}\rangle-\langle\eta,\{(r_{a})_{+}(\xi),(r_{a})_{+}(\zeta)\}_{c}\rangle.
\end{eqnarray*}

Thus we have
\begin{eqnarray*}
&&\langle\zeta,r_{+}(\xi\cdot_{r}\eta)-r_{+}(\xi)\cdot_{\g} r_{+}(\eta)+ [[r,r]](\xi,\cdot,\eta)\rangle\\
&=&\langle\eta,\{(r_s)_{+}(\xi),(r_s)_{+}(\zeta)\}_{c}\rangle-\langle\xi,(r_s)_{+}(\zeta)\cdot_{\g}(r_s)_{+}(\eta)\rangle
-\langle\eta,\{(r_a)_{+}(\zeta),(r_s)_{+}(\xi)\}_{c}\rangle\\
&&+\langle\xi,(r_a)_{+}(\zeta)\cdot_{\g}(r_s)_{+}(\eta)\rangle-\langle\zeta,(r_s)_{+}(\xi)\cdot_{\g}(r_s)_{+}(\eta)\rangle-\langle\zeta,(r_s)_{+}(\xi)\cdot_{\g}(r_a)_{+}(\eta)\rangle\\
&&-\langle\zeta,(r_a)_{+}(\xi)\cdot_{\g}(r_s)_{+}(\eta)\rangle-\langle\zeta,(r_a)_{+}(\xi)\cdot_{\g}(r_a)_{+}(\eta)\rangle+\langle\zeta,(r_{s})_{+}(\xi)\cdot_{\g}(r_{s})_{+}(\eta)\rangle\\
&&+\langle\xi,(r_{s})_{+}(\zeta)\cdot_{\g}(r_{s})_{+}(\eta)\rangle-\langle\eta,\{(r_{s})_{+}(\xi),(r_{s})_{+}(\zeta)\}_{c}\rangle-\langle\eta,\{(r_{a})_{+}(\xi),(r_{a})_{+}(\zeta)\}_{c}\rangle\\
 &=&\uwave{-\langle\eta,\{(r_a)_{+}(\zeta),(r_s)_{+}(\xi)\}_{c}\rangle}+\langle\xi,(r_a)_{+}(\zeta)\cdot_{\g}(r_s)_{+}(\eta)\rangle\uwave{-\langle\zeta,(r_s)_{+}(\xi)\cdot_{\g}(r_a)_{+}(\eta)\rangle}\\
 &&-\langle\zeta,(r_a)_{+}(\xi)\cdot_{\g}(r_s)_{+}(\eta)\rangle-\langle\zeta,(r_a)_{+}(\xi)\cdot_{\g}(r_a)_{+}(\eta)\rangle-\langle\eta,\{(r_{a})_{+}(\xi),(r_{a})_{+}(\zeta)\}_{c}\rangle\\
  &=&\langle\xi,\{(r_a)_{+}(\zeta),(r_s)_{+}(\eta)\}_{c}\rangle-\langle\xi,(r_s)_{+}(\eta)\cdot_{\g}(r_a)_{+}(\zeta)\rangle-\langle\zeta,\{(r_a)_{+}(\xi),(r_s)_{+}(\eta)\}_{c}\rangle\\
  &&+\langle\zeta,(r_s)_{+}(\eta)\cdot_{\g}(r_a)_{+}(\xi)\rangle-\langle\zeta,(r_a)_{+}(\xi)\cdot_{\g}(r_a)_{+}(\eta)\rangle-\langle\eta,\{(r_{a})_{+}(\xi),(r_{a})_{+}(\zeta)\}_{c}\rangle\\
  &\stackrel{\text{\eqref{skew-invariant}}}{=}&0,
\end{eqnarray*}
which implies that  \eqref{eq-homomorphism} holds. Thus, if $r$ satisfies $[[r,r]]=0$,
then $r_+$ is a mock-pre-Lie algebra homomorphism. Similarly, we can also prove that  $r_{-}$ is a mock-pre-Lie algebra homomorphism.

Conversely, if $(\g^{\ast},\cdot_{r})$ is a mock-pre-Lie algebra and $r_{+}$ is a mock-pre-Lie algebra homomorphism, by \eqref{eq-homomorphism}, we have $ [[r,r]]=0$.
\end{proof}

Next we introduce the notion of relative Rota-Baxter operators of weight $\lambda$ from a mock-pre-Lie algebra $\h$ to a mock-pre-Lie algebra $\g$ with respect to an action $(\rho,\mu).$
\begin{defi}
Let $(\g,\cdot_{\g})$ and $(\h,\cdot_{\h})$ be two mock-pre-Lie algebras.
Let two linear maps $\rho,\mu:\g \to \gl(\h)$ be a representation of the mock-pre-Lie algebra $\g$ on the vector space $\h$.
If for all $x\in \g, u,v\in \h,$
\begin{eqnarray}
\label{eq:mock-pre-action-1}{}\mu(x)(u\cdot_{\h}v)+u\cdot_{\h}(\mu(x)v)+v\cdot_{\h}(\mu(x)u)+\mu(x)(v\cdot_{\h}u)&=&0,\\
\label{eq:mock-pre-action-2}{}(\mu(x)u)\cdot_{\h}v+u\cdot_{\h}(\rho(x)v)+(\rho(x)u)\cdot_{\h}v+\rho(x)(u\cdot_{\h}v)&=&0,
\end{eqnarray}
then the pair $(\rho,\mu)$ is called  an {\bf action} of $(\g,\cdot_{\g})$ on $(\h,\cdot_{\h}).$
\end{defi}
\begin{ex}
Let $(\g,\cdot_{\g})$ be a mock-pre-Lie algebra. Then the regular representation $(\g;L,R)$ which is defined in Example \ref{regular-representation-mock-pre-Lie} is an action of $\g$ on itself.
\end{ex}

Let $(\rho,\mu)$ be an action of a mock-pre-Lie algebra $(\g,\cdot_{\g})$ on a mock-pre-Lie algebra $(\h,\cdot_{\h}).$ Define a new operation $\cdot_{(\rho,\mu)}$ on $\g\oplus\h$ by
\begin{eqnarray}
{}(x+u)\cdot_{(\rho,\mu)} (y+v)=x\cdot_{\g} y+\rho(x)v+\mu(y)u+u\cdot_{\h} v, \quad \forall x,y\in \g,u,v\in \h.
\end{eqnarray}

\begin{pro}\label{pro:semi}
 With above notations, $(\g\oplus \h,\cdot_{(\rho,\mu)})$ is a mock-pre-Lie algebra, which is called the semidirect product of the mock-pre-Lie algebra $\g$ and the mock-pre-Lie algebra $\h$ with respect to the action $(\rho,\mu).$
\end{pro}
\begin{proof}
  It follows from straightforward computations, and we omit details.
\end{proof}

\begin{defi}
Let $(\g,\cdot_{\g})$ and $(\h,\cdot_{\h})$ be two mock-pre-Lie algebras and $(\rho,\mu)$ be an action of $(\g,\cdot_{\g})$ on $(\h,\cdot_{\h})$. Then a linear map $B:\h\to \g$ is called a {\bf relative Rota-Baxter operator of weight $\lambda$} from a mock-pre-Lie algebra $\h$ to a mock-pre-Lie algebra $\g$ with respect to an action $(\rho,\mu)$ if $B$ satisfies the following equation:
\begin{equation}\label{rRBO}
(Bu)\cdot_{\g} (Bv)=B(\rho(Bu)v+\mu(Bv)u+\lambda u\cdot_\h v), \quad \forall u,v\in \h.
\end{equation}
\end{defi}

\begin{defi}
Let $(\g,\cdot_{\g})$ be a mock-pre-Lie algebra.
A linear map $B:\g\rightarrow\g$ is called a {\bf Rota-Baxter operator of weight $\lambda$} on a mock-pre-Lie algebra $(\g,\cdot_{\g})$ if
\begin{eqnarray*}
 (Bx)\cdot_{\g}(By)=B((Bx)\cdot_{\g}y+x\cdot_{\g}(By)+\lambda x\cdot_{\g}y),\quad \forall x,y\in \g.
\end{eqnarray*}
Moreover, a mock-pre-Lie algebra $(\g,\cdot_{\g})$ equipped with a Rota-Baxter operator $B$ is called a
Rota-Baxter mock-pre-Lie algebra of weight $\lambda$, which is denoted by $(\g,\cdot_{\g},B)$.
\end{defi}

Let $(\g,\cdot_{\g},B)$ be a Rota-Baxter mock-pre-Lie algebra of weight $\lambda$. Then there is a new mock-pre-Lie multiplication $\cdot_{B}$ on $\g$ defined by
\begin{eqnarray*}
x\cdot_{B} y=(Bx)\cdot_{\g} y+x\cdot_{\g} (By)+\lambda x\cdot_{\g}y.
\end{eqnarray*}
The mock-pre-Lie algebra $(\g,\cdot_{\g},B)$ is called the descendent mock-pre-Lie algebra and denoted by $\g_{B}$. Furthermore, $B$ is a mock-pre-Lie algebra homomorphism from the descendent mock-pre-Lie algebra $\g_{B}$ to $\g$.

\begin{rmk}
Note that a Rota-Baxter operator on a mock-pre-Lie algebra is a relative Rota-Baxter operator with respect to the regular action $(\g;L,R)$.
If the mock-pre-Lie algebra $(\h,\cdot_{\h})$ is abelian, then the action $(\rho,\mu)$ reduce to a representation, and the relative Rota-Baxter operator of weight $\lambda$ reduce to the relative Rota-Baxter operator of weight $0$ (also called an $\huaO$-operator). See \cite{Na-Attan} for more details.
\end{rmk}

By $I=r_{+}-r_{-},$ the mock-pre-Lie algebra multiplication $\cdot_{r}$ on $\g^*$ given by \eqref{dual-space-mockLie} reduces to
\begin{eqnarray}\label{new-dual-space-mockLie}
 \xi\cdot_{r} \eta=-(L^*_{r_{+}(\xi)}+R^*_{r_{+}(\xi)})\eta+R^*_{r_{+}(\eta)}\xi-R^*_{I(\eta)}\xi,  \quad \forall \xi,\eta \in \g^*.
\end{eqnarray}

Define a new multiplication $\cdot_+$ on $\g^*$ as follows:
\begin{equation}
\xi \cdot_+ \eta=R_{I(\eta)}^* \xi, \quad \forall \xi,\eta\in \g^*.
\end{equation}

\begin{pro}\label{cdot-mock-preLIE}
Let $(\g, \cdot_{\g})$ be a mock-pre-Lie algebra and $r\in \g\otimes
\g.$ If the skew-symmetric part $r_a$ of $r$ is $(L,L+R)$-invariant, then $(\g^*,\cdot_+)$ is a mock-pre-Lie algebra.
\end{pro}
\begin{proof}
By \eqref{defi-mock-preLie-equation} and Proposition \ref{invariance2}, for all $x\in \g,\xi,\eta,\zeta\in \g^*,$ we have
\begin{eqnarray*}
 &&\langle(\xi\cdot_{+}\eta)\cdot_{+}\zeta+\xi\cdot_{+}(\eta\cdot_{+}\zeta)+(\eta\cdot_{+}\xi)\cdot_{+}\zeta+\eta\cdot_{+}(\xi\cdot_{+}\zeta),x\rangle\\
 &=&\langle R^*_{I\zeta}R^*_{I\eta}\xi+R^*_{I(R^*_{I\zeta}\eta)}\xi+R^*_{I\zeta}R^*_{I\xi}\eta+R^*_{I(R^*_{I\zeta}\xi)}\eta,x\rangle\\
 &=&\langle\xi,(x\cdot_{\g}I\zeta)\cdot_{\g}I\eta-x\cdot_{\g}(I\eta \cdot_{\g}I\zeta)-\{x\cdot_{\g}I\zeta,I\eta\}_{c}-\{x,I\eta\}_{c}\cdot_{\g}I\zeta\rangle\\
  &=&0,
\end{eqnarray*}
which implies that $(\xi\cdot_{+}\eta)\cdot_{+}\zeta+\xi\cdot_{+}(\eta\cdot_{+}\zeta)=-(\eta\cdot_{+}\xi)\cdot_{+}\zeta-\eta\cdot_{+}(\xi\cdot_{+}\zeta)$. Thus $(\g^*,\cdot_+)$ is a mock-pre-Lie algebra.
\end{proof}

\begin{lem}
$(-L^*-R^*,R^*)$ is an action of the mock-pre-Lie algebra $(\g,\cdot_\g)$ on the mock-pre-Lie algebra $(\g^*,\cdot_+)$ given in Proposition \ref{cdot-mock-preLIE}.
\end{lem}

\begin{proof}
By \eqref{defi-mock-preLie-equation} and Proposition \ref{invariance2}, for all $x,y\in \g,\xi,\eta\in \g^*,$ we have
\begin{eqnarray*}
  &&\langle R^*_{x}(\xi\cdot_{+}\eta)+R^*_{x}(\eta\cdot_{+}\xi)+\xi\cdot_{+}(R^*_{x}\eta)+\eta\cdot_{+}(R^*_{x}\xi),y\rangle\\
  &=&\langle R^*_{x}R^*_{I\eta}\xi+R^*_{x}R^*_{I\xi}\eta+R^*_{I(R^*_{x}\eta)}\xi+R^*_{I(R^*_{x}\xi)}\eta,y\rangle\\
  &=&\langle\xi,(y\cdot_{\g}x)\cdot_{\g} I\eta-\{y\cdot_{\g}x,I\eta\}_{c}-y\cdot_{\g}(I\eta\cdot_{\g}x)-\{y,I\eta\}_{c}\cdot_{\g}x\rangle\\
  &=&0.
\end{eqnarray*}
From the above, equation \eqref{eq:mock-pre-action-1} holds, and  \eqref{eq:mock-pre-action-2} holds by direct computation. Consequently,
$(-L^*-R^*,R^*)$ is an action of the mock-pre-Lie algebra $(\g,\cdot_\g)$ on $(\g^*,\cdot_+).$
\end{proof}

Furthermore, we prove that quasi-triangular mock-pre-Lie bialgebras naturally induces a relative Rota-Baxter operator of weight $-1$
with respect to the coregular representation $(-L^*-R^*,R^*)$.

\begin{thm}\label{quasi-rRBO}
Let $(\g,\g_r^*)$ be a quasi-triangular mock-pre-Lie bialgebra induced by $r\in \g\otimes \g$. Then $r_+:(\g^*,\cdot_+) \to (\g,\cdot_\g)$ is a relative Rota-Baxter operators of weight $-1$ with
respect to the action $(-L^*-R^*,R^*)$.
\end{thm}
\begin{proof}
By Theorem \ref{heart} and \eqref{new-dual-space-mockLie}, for all $\xi,\eta\in \g^*,$ we have
 \begin{eqnarray*}
   (r_{+}\xi)\cdot_\g(r_{+}\eta)&=&r_{+}(\xi\cdot_{r}\eta)\\&=&r_{+}(-(L^*_{r_{+}(\xi)}+R^*_{r_{+}(\xi)})\eta+R^*_{r_{+}(\eta)}\xi-R^*_{I(\eta)}\xi)\\
    &=&r_{+}(-(L^*_{r_{+}(\xi)}+R^*_{r_{+}(\xi)})\eta+R^*_{r_{+}(\eta)}\xi-\xi\cdot_{+}\eta),
 \end{eqnarray*}
 which implies that $r_+:(\g^*,\cdot_+) \to (\g,\cdot_\g)$ is a relative Rota-Baxter operators of weight $-1$ with
respect to the action $(-L^*-R^*,R^*)$.
\end{proof}

\emptycomment{
\begin{cor}
 We can also define a new multiplication on $\g^*$ as follows:
$$ \xi \cdot_- \eta=L_{I\xi}^*\eta + R_{I\xi}^*\eta, \quad \forall \xi,\eta\in \g^*,$$
then $(\g^*,\cdot_-)$ is also a mock-pre-Lie algebra. We can prove that $r_-:(\g^*,\cdot_-) \rightarrow (\g,\circ_{\g})$ is a relative Rota-Baxter operators of weight $-1$ with respect to the action $(-L^*-R^*,R^*)$.
\end{cor}
\begin{proof}
 \begin{eqnarray*}
   (r_{-}\xi)\cdot_\g(r_{-}\eta)&=&r_{-}(\xi\cdot_{r}\eta)\\
   &=&r_{-}(-(L^*_{r_{-}(\xi)}+R^*_{r_{-}(\xi)})\eta+R^*_{r_{-}(\eta)}\xi)-L^*_{I(\xi)}\eta-R^*_{I(\xi)}\eta\\
    &=&r_{-}(-(L^*_{r_{-}(\xi)}+R^*_{r_{-}(\xi)})\eta+R^*_{r_{-}(\eta)}\xi-\xi\cdot_{-}\eta),
 \end{eqnarray*}
 which implies that $r_-:(\g^*,\cdot_-) \to (\g,\cdot_\g)$ is a relative Rota-Baxter operators of weight $-1$ with
respect to the action $(-L^*-R^*,R^*)$.
\end{proof}}

Taking a symmetric $r\in \g\otimes \g$ into Theorem \ref{quasi-rRBO}, then we have the following conclusion.

\emptycomment{
\Hou{Explanation}
\begin{rmk}
 Let $K:\g^*\rightarrow\g$ be a relative Rota-Baxter operator on $\g$ with respect to the representation $(-L^*-R^*,R^*)$. Then $\g^*_{K}:=(\g^*,\cdot_{K})$ is a mock-pre-Lie algebra, where $\cdot_{K}$ is given by
 \begin{eqnarray*}
   \xi\cdot_{K}\eta=-(L^*_{K(\xi)}+R^*_{K(\xi)})\eta+R^*_{K(\eta)}\xi,  \quad \forall \xi,\eta \in \g^*.
 \end{eqnarray*}
\end{rmk}}

\begin{cor}
Let $(\g,\g_r^*)$ be a triangular mock-pre-Lie bialgebra induced by a symmetric $r\in \g\otimes \g$. Then $r_+,r_-:\g^* \rightarrow\g$ are relative Rota-Baxter operators of weight $0$ with
respect to the coregular representation $(-L^*-R^*,R^*)$.
\end{cor}
\emptycomment{
\begin{cor}
  Let $(\g,\cdot_{\g})$ be a mock-pre-Lie algebra. Suppose that $r\in \g\otimes\g$ is symmetric. Then $r_{+}:\g^*\rightarrow\g$
  is a relative Rota-Baxter operator of weight $0$ on $(\g,\cdot_{\g})$ with respect to the coadjoint representation $(-L^*-R^*,R^*)$ if and only if $r$ is a solution of the classical Yang-Baxter equation in the mock-pre-Lie algebra.
\end{cor}
}

\section{Factorizable mock-pre-Lie bialgebras and quadratic Rota-Baxter mock-pre-Lie algebras}\label{sec:IV}
In this section, we introduce the notions of  factorizable mock-pre-Lie bialgebras and quadratic Rota-Baxter mock-pre-Lie algebras of weight $\lambda$.  We establish a one-to-one correspondence between factorizable mock-pre-Lie bialgebras and quadratic Rota-Baxter mock-pre-Lie algebras of nonzero weight. Moreover, we show that quadratic Rota-Baxter mock-pre-Lie algebras of weight $0$ can give rise a triangular mock-pre-Lie bialgebra.

\begin{defi}
A quasi-triangular mock-pre-Lie bialgebra $(\g,\g_{r}^*)$ is called {\bf factorizable} if the skew-symmetric part $r_a$ of $r$ is non-degenerate, which means that the linear map $I:\g^* \longrightarrow \g$ defined in \eqref{sign-I} is a linear isomorphism of vector spaces.
\end{defi}

Consider the map
$$
\g^*\stackrel{r_+\oplus r_-}{\longrightarrow}\g\oplus \g\stackrel{(x,y)\longmapsto x-y}{\longrightarrow}\g.
$$

\begin{pro}
  Let  $(\g,\g_{r}^{\ast})$ be a factorizable mock-pre-Lie bialgebra. Then $\Img(r_+\oplus r_-)$ is a mock-pre-Lie subalgebra of the direct sum mock-pre-Lie algebra $(\g \oplus \g,\cdot_{\oplus})$, which is given by Example \ref{direct-sum-mock-pre}, and $r_+\oplus r_-: (\g^*,\cdot_r)\rightarrow\Img(r_+\oplus r_-)$ is a mock-pre-Lie algebra isomorphism. Moreover, any $x\in \g$ has a unique decomposition
  \begin{equation}
    x=x_+-x_-,
  \end{equation}
  where $(x_+,x_-)\in \Img(r_+\oplus r_-)$.
\end{pro}
\begin{proof}
By Theorem \ref{heart}, we have both $r_{+}$ and $r_{-}$ are mock-pre-Lie algebra homomorphisms, then $\Img(r_{+})$ and $\Img(r_{-})$ are mock-pre-Lie subalgebra of $\g.$
Consider the direct sum mock-pre-Lie algebra $(\g \oplus \g,\cdot_{\oplus})$, therefore we have $\Img(r_+\oplus r_-)$ is a mock-pre-Lie subalgebra of $(\g \oplus \g,\cdot_{\oplus}).$ 
Since $I=r_+  -r_-:\g^*\to \g$ is nondegenerate, it follows that $r_+\oplus r_-: (\g^*,\cdot_r)\rightarrow\Img(r_+\oplus r_-)$ is a mock-pre-Lie algebra isomorphism.

  Since $I:\g^*\to \g$ is nondegenerate, then for all $x\in \g,$ there exist a unique $\xi \in \g^*,$ such that
  $I(\xi)=(r_{+}-r_{-})(\xi)=x,$ thus we have
  $$
 r_+ I^{-1}(x)-r_-I^{-1}(x)=(r_+  -r_-)I^{-1}(x)=x,
  $$
  which implies that $x=x_+-x_-,$ where $x_+=r_+ I^{-1}(x)$ and $x_-=r_-I^{-1}(x)$.
The uniqueness also follows from the fact that $I:\g^*\rightarrow \g$ is nondegenerate.
\end{proof}

{\rm Let $(\g,\g^*)$ be a mock-pre-Lie bialgebra. Define a multiplication $\cdot_{\frkd}$ on $\frkd=\g\oplus \g^*$ by
\begin{equation}
(x,\xi)\cdot_{\frkd}(y,\eta)=\Big(  x\cdot_\g y-(\huaL^*_{\xi}+\huaR^*_{\xi})y+\huaR^*_{\eta}(x),\xi\cdot_{\g^*} \eta-(L_x^*+R_x^*)\eta+R_y^*(\xi) \Big), \quad \forall x,y\in \g,\xi,\eta\in \g^*.
\end{equation}
By Proposition \ref{pro-matched-pair-mock-pre}, $(\frkd,\cdot_{\frkd})$ is a mock-pre-Lie algebra, which is called the {\bf mock-pre-Lie double of the mock-pre-Lie bialgebra} $(\g,\g^*),$ and denoted by $\frkd=\g \bowtie_{-L^*-R^*,R^*}^{-\huaL^*-\huaR^*,\huaR^*} \g^*$.}

\begin{thm}\label{mock-pre-double-fac}
Let $(\g,\g^*)$ be a mock-pre-Lie bialgebra. Suppose that $\{ e_1,e_2,\dots,e_n\}$ is a basis of $\g$ and $\{ e_1^*,e_2^*,\dots,e_n^*\}$ is the dual basis of $\g^*$. Then $r=\sum_{i} e_i\otimes e_i^* \in \g\otimes \g^*  \subset \frkd \otimes \frkd$ induces a mock-pre-Lie algebra structure on $\frkd^*=\g^*\oplus \g$ such that $(\frkd,\frkd_r^*)$ is a quasi-triangular mock-pre-Lie bialgebra, where the mock-pre-Lie algebra structure $(\frkd_r^*,\cdot_{\frkd_r^*})$ is given by
\begin{equation}\label{frkd-r}
(\xi,x)\cdot_{\frkd_r^*} (\eta,y)=(\xi \cdot_{\g^*} \eta, x \cdot_\g y),\quad \forall x,y\in \g,\xi,\eta\in \g^*.
\end{equation}
Moreover, $(\frkd,\frkd_r^*)$ is also a factorizable mock-pre-Lie bialgebra.
\end{thm}

\begin{proof}

We first prove that the skew-symmetric part $r_a=\frac{1}{2} \sum_{i} (e_i\otimes e_i^{\ast}-e_i^{\ast}\otimes e_i) $ of $r$ is $(\tilde{L},\tilde{L}+\tilde{R})$-invariant, where $\tilde{L},\tilde{R}$ are the left and right multiplication operators of the mock-pre-Lie algebra $(\frkd,\cdot_\frkd)$ respectively.  For all  $(\xi, x)\in \frkd^{\ast}$, we have $({r_a})_{+}(\xi, x)=\frac{1}{2} (-x,\xi)\in \frkd.$ Furthermore, by a direct calculation,  we have
\begin{eqnarray*}
\{(x,\xi),({r_a})_{+}(\eta,y)\}_{\frkc}
 &=& \frac{1}{2} \Big( -\{x,y\}_{c}-\huaL_{\eta}^*x+\huaL_{\xi}^*y,\{\xi,\eta\}_{c}-L_x^*\eta+L_y^* \xi \Big),\\
\tilde{L}^{\ast}_{(x,\xi)} (\eta,y)&=&\Big(-\{\xi,\eta\}_{c}+L_x^*\eta-L_y^* \xi,-\{x,y\}_{c}-\huaL_{\eta}^*x+\huaL_{\xi}^*y \Big).
\end{eqnarray*}
Thus we have
$$({r_a})_{+} \Big(\tilde{L}^{\ast}_{(x,\xi)}(\eta,y)\Big)+\{(x,\xi),({r_a})_{+}(\eta,y)\}_{\frkc}=0.$$
By Lemma \ref{skew-invariant-lem},  the skew-symmetric part $r_a$ of $r$ is  $(\tilde{L},\tilde{L}+\tilde{R})$-invariant. We also have
\begin{eqnarray*}
[[ r,r]] &=& \sum_{i,j}(-e_i\cdot_{\g} e_j\otimes e_j^* \otimes e_i^*-e_j^*\otimes e_i\cdot_{\g} e_j\otimes e_i^*+\{e_i, e_j^*\}_{\frkc}\otimes e_j \otimes e_i^*+e_i \otimes \{e_i^*,e_j\}_{\frkc} \otimes e_j^*\\
&&-e_i \otimes e_j\otimes \{e_i^*,e_j^*\}_{\frkc})\\
&=& \sum_{i,j}(-e_i\cdot_{\g} e_j\otimes e_j^* \otimes e_i^*-e_j^*\otimes e_i\cdot_{\g} e_j\otimes e_i^*-\huaL^*_{e_j^*} e_i\otimes e_j\otimes e_i^*-L^*_{e_i} e_j^*\otimes e_j\otimes e_i^*\\
&&-e_i\otimes L^*_{e_j} e_i^*\otimes e_j^*-e_i\otimes \huaL^*_{e_i^*} e_j\otimes e_j^*-e_i \otimes e_j\otimes \{e_i^*,e_j^*\}_{\frkc}).
\end{eqnarray*}
Note that
\begin{eqnarray*}
&&\sum_{i,j}(e_i\cdot_{\g} e_j\otimes e_j^* \otimes e_i^*+e_i\otimes L^*_{e_j} e_i^*\otimes e_j^*)=0;\\
&&\sum_{i,j}(e_j^*\otimes e_i\cdot_{\g} e_j\otimes e_i^*+L^*_{e_i} e_j^*\otimes e_j\otimes e_i^*)=0;\\
&&\sum_{i,j}(-\huaL^*_{e_j^*} e_i\otimes e_j\otimes e_i^*-e_i\otimes \huaL^*_{e_i^*} e_j\otimes e_j^*-e_i \otimes e_j\otimes \{e_i^*,e_j^*\}_{\frkc})=0.\\
\end{eqnarray*}
Thus $[[ r,r]] =0.$ Hence $(\frkd,\frkd^{\ast}_r)$ is a quasi-triangular mock-pre-Lie bialgebra.

Moreover, note that $r_{+},r_{-}:\frkd^{\ast} \longrightarrow
\frkd$ are given   respectively   by
$$ r_{+}(\xi,x)=(0,\xi),\quad r_{-}(\xi,x)=(x,0),\quad \forall x\in \g, \xi\in \g^*.$$
This implies that $I(\xi,x)=(-x,\xi)$, which means that the linear map $I:\frkd^{\ast} \longrightarrow \frkd$ is a linear isomorphism of vector spaces. Therefore, $(\frkd,\frkd^{\ast}_r)$ is a factorizable mock-pre-Lie bialgebra.
\end{proof}

\begin{defi}
Let $(\g, \cdot_{\g},B)$ be a Rota-Baxter mock-pre-Lie algebra of weight $\lambda$ and $(\g, \cdot_{\g},\omega)$ a quadratic mock-pre-Lie algebra. Then the quadruple $(\g, \cdot_{\g},B,\omega)$ is called a {\bf quadratic Rota-Baxter mock-pre-Lie algebra of weight $\lambda$} if the following compatibility condition holds:
\begin{eqnarray}\label{quadratic-Rota-mock-pre}
\omega(Bx,y)+\omega(x,By)+\lambda\omega(x,y)=0,\quad \forall x,y\in \g.
\end{eqnarray}
\end{defi}

The following theorem shows that a factorizable mock-pre-Lie bialgebra naturally gives rise to a quadratic Rota-Baxter mock-pre-Lie algebra of nonzero weight.

\begin{thm}\label{Factorizable-mock-Lie-algebra}
Let $(\g,\g_{r}^*)$ be a factorizable mock-pre-Lie bialgebra with $I=r_{+}-r_{-}$. Then $(\g,B,\omega_{I})$ is a quadratic Rota-Baxter mock-pre-Lie algebra of weight $\lambda$, where the linear map $B:\g \rightarrow \g$ and $\omega_{I}\in \wedge^{2} \g^*$ are defined respectively by
\begin{eqnarray}
\label{B}B&=&\lambda r_{-}\circ I^{-1},\\
\label{SI}\omega_I(x,y)&=&\langle I^{-1}x,y \rangle, \quad \forall x,y \in \g.
\end{eqnarray}
\end{thm}

\begin{proof}
Since $r_{+},r_{-}:(\g^*,\cdot_{r}) \longrightarrow (\g,\cdot_{\g})$ are both mock-pre-Lie algebra homomorphisms, for all $x,y\in \g$, we have
\begin{eqnarray}\label{RB operator in fac}
I(I^{-1}x \cdot_{r} I^{-1}y )&=&(r_{+}-r_{-})(I^{-1}x \cdot_{r} I^{-1}y )\\
\nonumber&=&((I+r_{-})(I^{-1}x \cdot_{r} I^{-1}y )-r_{-}(I^{-1}x \cdot_{r} I^{-1}y )\\
\nonumber&=&((I+r_{-})I^{-1}x) \cdot_{\g} ((I+r_{-})I^{-1}y)- (r_{-}I^{-1}x)\cdot_{\g} (r_{-}I^{-1}y)\\
\nonumber&=& (r_{-}I^{-1}x) \cdot_{\g} y+x \cdot_{\g} (r_{-}I^{-1}y)+x \cdot_{\g} y.
\end{eqnarray}
Therefore, we have
\begin{eqnarray*}
B(x)\cdot_{\g}B(y)&=&\lambda^{2}(r_{-}I^{-1}x \cdot_{\g} r_{-}I^{-1}y )\\
&=&\lambda^{2}r_{-}(I^{-1}x \cdot_{r} I^{-1}y )\\
&=& \lambda^{2} r_{-}I^{-1}\Big(x \cdot_{\g} (r_{-}I^{-1}y)+(r_{-}I^{-1}x)\cdot_{\g} y+x \cdot_{\g} y \Big)\\
&=&B(B(x)\cdot_{\g} y+x\cdot_{\g}B(y) +\lambda x\cdot_{\g} y),
\end{eqnarray*}
which implies that $B$ is a Rota-Baxter operator of weight $\lambda$ on $\g$.

Next we show that $(\g,\cdot_{\g},B,\omega_{I})$ is a quadratic Rota-Baxter mock-pre-Lie algebra. Since $I^{*}=-I$, we have
$$ \omega_{I}(x,y)=\langle I^{-1}x,y \rangle=-\langle x,I^{-1}y \rangle=-\omega_{I}(y,x),$$
 which means that $ \omega_{I}$ is skew-symmetric.

Since the skew-symmetric part $r_{a}$ of $r$ is $(L,L+R)$-invariant, by Proposition \ref{invariance2},
we can deduce that $ I^{-1} \circ L_x =-(L_x^*+R_x^*)\circ I^{-1}$. Thus
\begin{eqnarray*}
\omega_{I}(x\cdot_{\g} y,z)-\omega_{I}(y,\{x,z\}_{c})
&=&\langle I^{-1}(x\cdot_{\g} y),z  \rangle-\langle I^{-1}y,\{x,z\}_{c} \rangle\\
&=&\langle I^{-1}\circ(L_{x}y)+(L^*_{x}+R^*_{x}) \circ (I^{-1}y),z \rangle=0,
\end{eqnarray*}
which implies that \eqref{quadratic-mock-pre} holds.

Moreover, by using $r^{\ast}_{-}=r_{+}$ and $I=r_{+}-r_{-}$, we have
\begin{eqnarray*}
\omega_{I}(x,By)+\omega_{I}(Bx,y)+\lambda \omega_{I}(x,y)&=&\lambda\Big(\langle I^{-1}(x), r_{-}I^{-1}(y) \rangle+\langle  I^{-1} r_{-} I^{-1}(x),y \rangle+\langle I^{-1}x,y \rangle \Big)\\
&=& \lambda\langle(-I^{-1} r_{+} I^{-1}+I^{-1} r_{-} I^{-1}+ I^{-1})(x),y \rangle=0,
\end{eqnarray*}
which implies that \eqref{quadratic-Rota-mock-pre} holds.

Therefore, $(\g,\cdot_{\g},B,\omega_{I})$ is a quadratic Rota-Baxter mock-pre-Lie algebra of weight $\lambda$.
\end{proof}

At the end of this section, we show that a quadratic Rota-Baxter mock-pre-Lie algebra of weight $\lambda$ can give rise to a triangular  mock-pre-Lie bialgebra and  factorizable mock-pre-Lie bialgebra, respectively .

\begin{thm}\label{thm:QRB-tri and fac-mock-pre}
    Let $(\g,B,\omega)$ be a quadratic Rota-Baxter mock-pre-Lie algebra of weight $\lambda$  and
    $ \huaI_{\omega}: \g^*\longrightarrow \g $ the induced linear isomorphism given by $\langle\huaI_{\omega}^{-1}x,y \rangle :=\omega(x,y).$ Then $r \in \g\otimes \g $ defined by
  \begin{equation}\label{eq:equiv} r^{B,\omega}_{+}:= B\circ \huaI_{\omega}:\g^* \rightarrow \g, \quad
r^{B,\omega}_{+}(\xi)=r^{B,\omega}(\xi,\cdot), \quad \forall \xi\in \g^*,
\end{equation}
  satisfies $[[r,r]]=0$, and gives rise to a triangular (if $\lambda=0$) or a factorizable (if $\lambda\neq 0$) mock-pre Lie bialgebra $(\g,\g_{r^{B,\omega}}^*)$.
\end{thm}
\begin{proof}
    Since $\omega$ is skew-symmetric, let $\huaI_{\omega}^{-1}x=\xi,$ we have
 \begin{eqnarray*}
  \langle \huaI_{\omega}^*\eta,\xi\rangle=\langle \eta,\huaI_{\omega}(\xi)\rangle
  =\omega(\huaI_{\omega}(\eta),\huaI_{\omega}(\xi))=-\omega(\huaI_{\omega}(\xi),\huaI_{\omega}(\eta))=-\langle \xi,\huaI_{\omega}(\eta)\rangle,
 \end{eqnarray*}
   which means $\huaI_{\omega}=-\huaI_{\omega}^{\ast}$.

{\bf Case (1):} When $\lambda=0,$ by the fact that $\omega(x,By)+\omega(Bx,y)=0$  for all $x,y\in \g$, we have
    $$ \langle \huaI_{\omega}^{-1}x,By\rangle+\langle \huaI_{\omega}^{-1}\circ B(x),y \rangle=0, $$
    which implies that $B^*\circ \huaI_{\omega}^{-1}+\huaI_{\omega}^{-1}\circ B=0,$ and then
    $$ \huaI_{\omega}\circ B^*+B\circ \huaI_{\omega}=0.$$
We have
    $$ r^{B,\omega}_{-}:=(r^{B,\omega}_{+})^{\ast}=-\huaI_{\omega}\circ B^*= B\circ \huaI_{\omega}=r^{B,\omega}_{+}. $$
  Thus $r^{B,\omega}\in \g\otimes\g$ is symmetric. Define a multiplication $\cdot_{r^{B,\omega}}$ on $\g^*$ by
  \begin{eqnarray}\label{dualspace-mock-preLie}
   \xi \cdot_{r^{B,\omega}} \eta= -(L_{r^{B,\omega}_{+}(\xi)}^*+R_{r^{B,\omega}_{+}(\xi)}^*)\eta+R_{r^{B,\omega}_{-}(\eta)}^* \xi.
  \end{eqnarray}
By the fact that $\omega(x\cdot_{\g} y,z)-\omega(y,\{x,z\}_{c})=0$ for all $x,y,z\in \g$, we have
$$ \langle \huaI_{\omega}^{-1}\circ L_x(y),z \rangle+\langle (L_{x}^*+R_{x}^*)\circ \huaI_{\omega}^{-1}(y),z \rangle=0, $$
which implies that  $ L_x\circ \huaI_{\omega}=-\huaI_{\omega} \circ (L_{x}^*+R_{x}^*).$ For all $\xi,\eta\in \g^*$, we have
\begin{eqnarray*}
    \huaI_{\omega}(\xi\cdot_{r^{B,\omega}} \eta)&=&\huaI_{\omega}\Big(-(L_{r^{B,\omega}_+(\xi)}^*+R_{r^{B,\omega}_+(\xi)}^*)\eta+R_{r^{B,\omega}_{-}(\eta)}^* \xi \Big)\\
    &=& L_{r^{B,\omega}_{+}(\xi)} \circ \huaI_{\omega}(\eta)+\huaI_{\omega}\circ R_{r^{B,\omega}_{-}(\eta)}^{\ast}(\xi)\\
    &=& L_{r^{B,\omega}_{+}(\xi)} \circ \huaI_{\omega}(\eta)+ R_{r^{B,\omega}_{-}(\eta)}\circ \huaI_{\omega}(\xi)\\
    &=&r^{B,\omega}_+(\xi)\cdot_{\g}\huaI_{\omega}(\eta) +\huaI_{\omega}(\xi)\cdot_{\g} r^{B,\omega}_+(\eta)\\
     &=&(B\huaI_{\omega}\xi)\cdot_{\g}(\huaI_{\omega}\eta)+(\huaI_{\omega}\xi)\cdot_{\g}(B\huaI_{\omega}\eta)\\
     &=&(\huaI_{\omega}\xi)\cdot_{B}(\huaI_{\omega}\eta),
\end{eqnarray*}
which implies that
\begin{eqnarray*}
        \huaI_{\omega}(\xi \cdot_{r^{B,\omega}} \eta)=(\huaI_{\omega}\xi)\cdot_B(\huaI_{\omega}\eta).
 \end{eqnarray*}
 Thus $\cdot_{r^{B,\omega}}$ is a mock-pre-Lie algebra multiplication and $\huaI_{\omega}$ is a mock-pre-Lie algebra isomorphism from $(\g^*,\cdot_{r^{B,\omega}})$ to $(\g,\cdot_B)$. By $B:(\g,\cdot_B)\longrightarrow (\g,\cdot_{\g})$ is a mock-pre-Lie algebra homomorphism, we have
$$   r^{B,\omega}_{+}=r^{B,\omega}_{-}=B\circ \huaI_{\omega}:(\g^*,\cdot_{r^{B,\omega}}) \longrightarrow (\g,\cdot_{\g})$$
    are both mock-pre-Lie algebra homomorphisms. Therefore, by Theorem \ref{heart}, we have $ [[r,r]] =0$ and $(\g,\g_{r^{B,\omega}}^*)$ is a triangular mock-pre-Lie bialgebra.

{\bf Case (2):} When $\lambda\neq0,$ by the fact that $\omega(x,By)+\omega(Bx,y)+\lambda \omega(x,y)=0$  for all $x,y\in \g$, we have
    $$ \langle \huaI_{\omega}^{-1}x,By\rangle+\langle \huaI_{\omega}^{-1}\circ B(x),y \rangle+\lambda\langle \huaI_{\omega}^{-1}x,y \rangle=0, $$
    which implies that $B^*\circ \huaI_{\omega}^{-1}+\huaI_{\omega}^{-1}\circ B+\lambda \huaI_{\omega}^{-1}=0, $ and then
    $$ \huaI_{\omega}\circ B^*+B\circ \huaI_{\omega}+\lambda \huaI_{\omega}=0.$$
Thus we have
    $$ r^{B,\omega}_{-}:=(r^{B,\omega}_{+})^{\ast}=-\huaI_{\omega}\circ B^*=B\circ \huaI_{\omega}+\lambda \huaI_{\omega}, $$
    and $-\lambda\huaI_{\omega}=r^{B,\omega}_{+}-r^{B,\omega}_{-}$. Define a multiplication $\cdot_{r^{B,\omega}} $ on $\g^*$ by \eqref{dualspace-mock-preLie}.
By the fact that $\omega(x\cdot_{\g} y,z)-\omega(y,\{x,z\}_{c})=0$ for all $x,y,z\in \g$, we have
$$ \langle \huaI_{\omega}^{-1}\circ L_x(y),z \rangle+\langle (L_{x}^*+R_{x}^*)\circ \huaI_{\omega}^{-1}(y),z \rangle=0, $$
which implies that  $ L_x\circ \huaI_{\omega}=-\huaI_{\omega} \circ (L_{x}^*+R_{x}^*).$ Therefore, by Proposition
\ref{invariance2}, the skew-symmetric part $r_a$ of $r$ is $(L,L+R)$-invariant.
By the same method as in Case (1), we have
\begin{eqnarray*}
         \huaI_{\omega}(\xi \cdot_{r^{B,\omega}} \eta)=(\huaI_{\omega}\xi)\cdot_B(\huaI_{\omega}\eta).
 \end{eqnarray*}
 Thus $\cdot_{r^{B,\omega}}$ is a mock-pre-Lie algebra multiplication and $\huaI_{\omega}$ is a mock-pre-Lie algebra isomorphism from $(\g^*,\cdot_{r^{B,\omega}})$ to $(\g,\cdot_B)$. Finally, by the fact that $B,B+\lambda{\Id}:(\g,\cdot_B)\rightarrow (\g,\cdot_{\g})$ are both mock-pre-Lie algebra homomorphisms, we deduce that$$   r^{B,\omega}_{+}:= B\circ \huaI_{\omega},  \quad r^{B,\omega}_{-}:=(B+\lambda {\Id})\circ \huaI_{\omega}:(\g^*,\cdot_{r^{B,\omega}}) \longrightarrow (\g,\cdot_{\g})$$
    are both mock-pre-Lie algebra homomorphisms. Therefore, by Theorem \ref{heart}, we have $ [[r,r]] =0$ and $(\g,\g_{r^{B,\omega}}^*)$ is a quasi-triangular mock-pre-Lie bialgebra. Since $\huaI_{\omega} = (-\frac{1}{\lambda})(r^{B,\omega}_{+}-r^{B,\omega}_{-})$  is an isomorphism, the mock-pre-Lie bialgebra $(\g,\g_{r^{B,\omega}}^*)$ is factorizable.
\end{proof}

 \end{document}